\documentclass{amsart}
\usepackage{amsfonts}

\usepackage{amsmath}

\usepackage{amssymb}
\usepackage{amsthm}
\usepackage{mathrsfs}
\usepackage{mathtools}
\usepackage{esint}
\newtheorem{theorem}{Theorem}[section]
\newtheorem{corollary}{Corollary}[section]
\newtheorem{lemma}{Lemma}[section]
\newtheorem{proposition}{Proposition}[section]
\theoremstyle{definition}
\newtheorem{definition}{Definition}[section]
\theoremstyle{remarq}

\theoremstyle{remark}
\newtheorem{remark}{Remark}[section]
\theoremstyle{nota}

\theoremstyle{notation}

\theoremstyle{exemple}

\theoremstyle{example}

\numberwithin{equation}{section}

\def\XXint#1#2#3{{\setbox0=\hbox{$#1{#2#3}{int}$}  
		\vcenter{\hbox{$#2#3$}}\kern-.5\wd0}}

\usepackage{graphicx}
\usepackage{color}
\usepackage{hyperref}

\begin{document}
	\title[Stochastic-Periodic Homogenization in Orlicz-Sobolev Spaces]{Stochastic-Periodic Homogenization of a Class of Minimization Problems in Orlicz-Sobolev Spaces}
	
		\author{Fotso Tachago Joel}
	\address{J. Fotso Tachago, University of Bamenda, 	Higher Teachers Training College, Department of Mathematics
		P.O. Box 39
		Bambili, Cameroon}
	\email{fotsotachago@yahoo.fr}
	
	\author{Dongho Joseph}
	\address{J. Dongho, University of Maroua, Department of Mathematics and Computer Science, P.O. box 814, Maroua, Cameroon}
	\email{joseph.dongho@fs.univ-maroua.cm}

	\author{Tchinda Takougoum Franck}
	\address{F. Tchinda Takougoum, University of Maroua, Department of Mathematics and Computer Science, P.O. box 814, Maroua, Cameroon}
	\email{takougoumfranckarnold@gmail.com}
	
	\date{Otober, 2025}
	\subjclass[2020]{35B27, 35B40, 49J55,  37A05, 46E30}
	\keywords{Homogenization, Stochastic-Periodic, Dynamical systems, Minimization Problems, Stochastic Two-scale Convergence,  Orlicz-Sobolev Spaces}
	
	\begin{abstract}
			We develop the stochastic two-scale convergence method in the framework of Orlicz-Sobolev spaces, in order to deal with the homogenization of coupled stochastic-periodic problems in such spaces. One fundamental in this topic is the extension of compactness results for this method to the Orlicz setting. For the application, we show that the sequence of minimizers of a class of highly oscillatory minimizations problems involving integral functionals with convex and nonstandard growth integrands, converges to the minimizer of a homogenized problem.
	\end{abstract}
	
	\maketitle


\section{Introduction} \label{sect1} 

This paper is devoted to the  homogenization theory in the framework of Orlicz-Sobolev spaces, and more precisely to the \textit{stochastic two-scale convergence} method (see, e.g., \cite{sango}) in this type of spaces. So, we consider the following minimization problem
\begin{equation}\label{ras2}
	\min \left\{ F_{\epsilon}(v) : v \in W^{1}_{0}L^{\Phi}_{D_{x}}(Q; L^{\Phi}(\Omega)) \right\},
\end{equation} 
where, for each $\epsilon>0$ the functional $F_{\epsilon}$ is defined by 
\begin{eqnarray}\label{ras1}
	F_{\epsilon}(v) = \iint_{Q\times \Omega} f\left(x, T\left(\frac{x}{\epsilon}\right)\omega, \frac{x}{\epsilon^{2}}, \, Dv(x,\omega) \right)dxd\mu,  \\ \quad v \in W^{1}_{0}L^{\Phi}_{D_{x}}(Q; L^{\Phi}(\Omega)), \nonumber
\end{eqnarray} 
with $f$ a random-periodic integrand (see, Subsection \ref{hypoproblem} and comment below) and we are interested in the  homogenization (i.e., the analysis of  asymptotic behaviour when $0<\epsilon \rightarrow 0$ of the sequence of solutions) of problem (\ref{ras2})-(\ref{ras1}). The study of this homogenization problem under random and periodic considerations of the integrand $f$ is so-called stochastic-periodic homogenization.

\subsection{Our hypotheses}\label{hypoproblem}

Let us specify data in (\ref{ras2})-(\ref{ras1}). We refer to Section \ref{sect2} for general notations.
\begin{itemize}
	\item $\Phi : [0,+\infty[ \rightarrow [0,+\infty[ $ be a Young function of class   $\Delta_{2}\cap \Delta'$, such that its  conjugate $\widetilde{\Phi}$ is also a Young function of class $\Delta_{2}\cap \Delta'$ ;
	\item $(\Omega, \mathscr{M}, \mu)$ be a measure space with probability measure $\mu$ ;
	\item $Q$ is a bounded open set in $\mathbb{R}^{N}_{x}$ (the space $\mathbb{R}^{N}$ of variables $x = (x_{1}, \cdots , x_{N})$) and $\{ T(x) : \Omega \rightarrow \Omega, \, x \in \mathbb{R}^{N} \}$ is a fixed $N$-dimensional dynamical system on $\Omega$, with respect to a fixed probability measure $\mu$ on $\Omega$ invariant for $T$ ;
	\item $D$ denotes the usual gradient operator in $Q$, i.e., $D = D_{x}= \left(\frac{\partial}{\partial x_{i}} \right)_{1\leq i\leq N}$ and $W^{1}_{0}L^{\Phi}_{D_{x}}(Q; L^{\Phi}(\Omega))$ is an Orlicz-Sobolev space which will be specified later (see, (\ref{c1orli1S})) ;
	\item $f : \mathbb{R}^{N}\times\Omega\times\mathbb{R}^{N}\times \mathbb{R}^{N}\rightarrow\mathbb{R}$, $(x,\omega,y,\lambda) \rightarrow f(x,\omega,y,\lambda)$ is a random-periodic integrand satisfying the following properties:
	\begin{itemize}
		\item[(H$_{1}$)] \textit{(Carath\'{e}odory function hypothesis)} : for all $(x,\lambda) \in \mathbb{R}^{N}\times \mathbb{R}^{N}$ and for almost all $(\omega,y)\in\Omega\times\mathbb{R}^{N}$, $f(x,\cdot,\cdot, \lambda)$ is measurable and $f(\cdot,\omega,\cdot,\cdot)$ is continuous. Moreover, there exist a continuous positive function $\varpi : \mathbb{R} \to \mathbb{R}_{+}$ with $\varpi(0)=0$, and a function $a \in L^{1}_{loc}(\mathbb{R}^{N}_{y})$ such that 
		\begin{equation}
			|f(x,\omega,y,\lambda) - f(x',\omega,y,\lambda)| \leq \varpi(|x - x'|)[a(y) + f(x,\omega,y,\lambda)],
		\end{equation}
		for all $x, x' \in \mathbb{R}^{N}$, $\lambda \in \mathbb{R}^{N}$ and for $d\mu\times dy$-almost all $(\omega,y) \in \Omega\times\mathbb{R}^{N}$ ; 
		\item[(H$_{2}$)] \textit{(Strictly convexity hypothesis) } :$f(x,\omega,y,\cdot)$ is strictly convex for $d\mu\times dy$-almost all $(\omega,y) \in \Omega\times\mathbb{R}^{N}$ and for all $x\in \mathbb{R}^{N}$ ;
		\item[(H$_{3}$)] \textit{(Nonstandard growth hypothesis)} : There are two constants $c_{1}, \, c_{2} > 0$ such that 
		\begin{equation}
			c_{1} \Phi(|\lambda|) \leq f(x,\omega,y,\lambda) \leq c_{2}(1+ \Phi(|\lambda|)),
		\end{equation}
		for all $(x,\lambda) \in \mathbb{R}^{N}\times \mathbb{R}^{N}$ and for $d\mu\times dy$-almost all $(\omega,y) \in \Omega\times\mathbb{R}^{N}$ 
		\item[(H$_{4}$)] \textit{(Periodic structure hypothesis)} : $f(x,\omega,\cdot,\lambda)$ is periodic for all $(x,\lambda) \in \mathbb{R}^{N}\times \mathbb{R}^{N}$ and for almost all $\omega \in \Omega$, i.e.,
		\begin{equation}
			\textup{for\; each} \;\, k \in \mathbb{Z}^{N}, \; f(x, \omega, y+k, \lambda) = f(x, \omega, y, \lambda), \;\; \textup{for\; all} \;\, y \in \mathbb{R}^{N}.
		\end{equation}  
	\end{itemize}
\end{itemize}

\subsection{Literature review}

With considerations to the principles of physics, for example in elasticity theory (see \cite{gambi}), the term $F_{\epsilon}(v)$ can be viewed as the energy under a deformation $v$ of an elastic body whose microstructures behave randomly in some scales and periodically in some other scales. 
Let us already note that, the minimization problems of the form \eqref{ras2}-\eqref{ras1} have been studied by many authors in Orlicz setting, taking into account only the stochastic aspect or periodic aspect, separately. So, in \cite{tacha1}, J. Fotso Tachago and H. Nnang after having extended the periodic two-scale convergence (see, e.g., \cite{allair1,nguet1}) to Orlicz setting, will study the periodic homogenization of problem \eqref{ras2}-\eqref{ras1}, where the functionals $F_{\epsilon}$ are under the form $\int_{Q} f(\frac{x}{\epsilon}, Dv(x))\,dx$, with $v \in W^{1}_{0}L^{\Phi}(Q)$. In \cite{tacha3}, the same authors in order to tackle with periodic multiscale problems in Orlicz setting, will extend the reiterated periodic  two-scale convergence method (see, e.g., \cite{allair3}) and will study the periodic homogenization of problem \eqref{ras2}-\eqref{ras1}, where the functionals $F_{\epsilon}$ are under the form $\int_{Q} f(\frac{x}{\epsilon}, \frac{x}{\epsilon^{2}}, Dv(x))\,dx$ with $v \in W^{1}_{0}L^{\Phi}(Q)$.
In \cite{martin}, M. Kalousek  will use the periodic two-scale convergence (extended to  Orlicz's spaces by J. Fotso Tachago and H. Nnang \cite{tacha1}), for the homogenization of incompressible generalized Stokes flows through a porous medium. For others works in related with periodic homogenization in  Orlicz-Sobolev spaces, we refer to \cite{tacha2,tacha5,tacha4,ttchin1}. On the other hand, J. Dongho \textit{et al} \cite{franck}, after having extended to Orlicz setting the stochastic two-scale convergence in the mean (see \cite{Bourgeat}), will study the stochastic homogenization of problem \eqref{ras2}-\eqref{ras1}, where the functionals $F_{\epsilon}$ are under the form $\iint_{Q\times \Omega} f\left(T\left(\frac{x}{\epsilon}\right)\omega,\, Dv(x,\omega) \right)dxd\mu$ with $v \in W^{1}_{0}L^{\Phi}_{D_{x}}(Q; L^{\Phi}(\Omega))$.

However, all the previously cited works take into account either periodic aspects only, or stochastic aspects  in the homogenization process; which is not always consistent with most physical phenomena since most of these phenomena behave randomly in some scales, and periodically in others scales. Moreover, a scale can not be at the same time periodic and random. For instance, it is known that the human body is an example of medium that presents both a random and deterministic (or periodic) behaviour. In particular, it contains an exciting class of nonlinear materials presenting microstructures like myocardium, arterials walls, cartilage, muscles, etc. 
Thus, in \cite{sango}, M. Sango and J.L. Woukeng propose a general method which they call \textit{stochastic two-scale convergence} and which will be used later to deal with coupled stochastic-periodic homogenization problems in the framework of classical Lebesgue spaces.
In addition, they study the homogenization of problem \eqref{ras2}-\eqref{ras1}, where the functionals $F_{\epsilon}$ are under the form  $\iint_{Q\times\Omega} f(x,T(\frac{x}{\epsilon})\omega, \frac{x}{\epsilon^{2}}, Dv(x,\omega))\,dx\,d\mu$ with $v \in L^{p}(\Omega, W_{0}^{1,p}(Q))$. In \cite{joelf}, J. Fotso Tachago and H. Nnang use this stochastic two-scale convergence for the homogenization of Maxwell equations with linear and periodic conductivity.  For others works on stochastic homogenization and periodic homogenization in classical Sobolev's spaces we refer, e.g., to \cite{abda,andre,bia,blan,cham,dal,gambi,jikov}.

\subsection{Motivation and objectives}

As mentioned at the beginning of the introduction, our main objective in this paper is to homogenize (in Orlicz setting) the 
minimization problem (\ref{ras2})-(\ref{ras1}), under the hypotheses of Subsection \ref{hypoproblem}. But, first of all we have need to extend the \textit{stochastic two-scale convergence} method (see, \cite{sango}) to the framework of Orlicz-Sobolev spaces. To the best our knowledge, the homogenization problem (\ref{ras2})-(\ref{ras1})  has not been address in Orlicz setting, thus making this work original.  
The extension to the Orlicz-Sobolev spaces is  motivated by the fact that the stochastic two-scale convergence method introduced by M. Sango and J.L. Woukeng \cite{sango} have been widely adopted in stochastic-periodic homogenization of partial differential equations in classical Sobolev spaces neglecting materials where microstructure cannot be conveniently captured by modeling exclusively by means of those spaces.  Moreover, it is well known (see, e.g. \cite{martin,mignon2}) that there exist most problems whose solutions must naturally belong not to the classical Sobolev spaces, but rather to the Orlicz-Sobolev spaces. 

Note that, this homogenization method whose we extend to Orlicz setting is a combination of  both well-known schemes, periodic two-scale convergence [see, J. Fotso Tachago and H. Nnang: Two-scale convergence of integral functionals
with convex, periodic and nonstandard growth integrands, \textit{Acta applicandae
	mathematicae}, \textbf{121}:175--196,2012] and stochastic two-scale convergence in the mean [see, J. Dongho, J. Fotso Tachago, and F. Tchinda Takougoum: Stochastic twos-cale convergence in the mean in orlicz-sobolev's spaces and applications to the
homogenization of an integral functional, \textit{Asymptotic Analysis}, \textbf{142}(1):291--320,
2025]. 

\subsection{Main results}

In the present paper, a first novelty is concerned the generalization of  compactness results for stochastic two-scale convergence (see, e.g. \cite{sango}) to  Orlicz-Sobolev spaces.
At the same time, the solutions of the  homogenization problem \eqref{ras2}-\eqref{ras1} under consideration seems to be more general than the case considered in \cite{sango}.
Thus, considering the notations in Section \ref{sect2}, we prove the following compactness results.

\begin{theorem}(compactness 1)\label{lem4} \\
	Let $\Phi \in \Delta_{2}$ be a Young function and let $E$ be a \textit{fundamental sequence} (see, Section \ref{sect3}).	Then, any bounded sequence $(u_{\epsilon})_{\epsilon\in E}$  in $L^{\Phi}(Q\times\Omega)$ admits a subsequence which is weakly stochastically 2-scale convergent in $ L^{\Phi}(Q\times\Omega)$.
\end{theorem}

\begin{theorem}(Compactness 2)\label{lem29} \\
	Let $\Phi \in \Delta_{2}$ be a Young function and $\widetilde{\Phi} \in \Delta_{2}$ its conjugate. Let $E$ be a \textit{fundamental sequence}.
	Assume $(u_{\epsilon})_{\epsilon\in E}$ is a sequence in $W^{1}L^{\Phi}_{D_{x}}(Q; L^{\Phi}(\Omega))$ such that : 
	\begin{itemize}
		\item[(i)] $(u_{\epsilon})_{\epsilon\in E}$ is bounded in $L^{\Phi}\left(Q\times\Omega\right)$ and $(D_{x}u_{\epsilon})_{\epsilon\in E}$ is bounded in $L^{\Phi}\left(Q\times\Omega\right)^{N}$.
	\end{itemize}
	Then there exist $u_{0} \in W^{1}L^{\Phi}_{D_{x}}(Q; I_{nv}^{\Phi}(\Omega))$, $u_{1} \in L^{1}\left(Q; W^{1}_{\#}L^{\Phi}(\Omega)\right)$ with $\overline{D}_{\omega}u_{1} \in L^{\Phi}(Q\times\Omega)^{N}$, $u_{2}\in L^{1}(Q\times\Omega; W^{1}_{\#}L^{\Phi}_{per}(Y))$ with $D_{y}u_{2} \in L^{\Phi}(Q\times\Omega\times Y_{per})^{N}$ and a subsequence $E'$ from $E$ such that, as $E' \ni \epsilon \rightarrow 0$, 
	\begin{itemize}
		\item[(ii)] $u_{\epsilon} \rightarrow u_{0}$ stoch. in $L^{\Phi}(Q\times \Omega)$-weak 2s;
		\item[(iii)] $D_{x}u_{\epsilon} \rightarrow D_{x}u_{0} + \overline{D}_{\omega}u_{1} + D_{y}u_{2}$ stoch. in $L^{\Phi}(Q\times \Omega)^{N}$-weak 2s,  with $u_{1} \in L^{\Phi}(Q; W^{1}_{\#}L^{\Phi}(\Omega))$, $u_{2}\in L^{\Phi}(Q\times\Omega; W^{1}_{\#}L^{\Phi}_{per}(Y))$  when $\widetilde{\Phi} \in \Delta'$.  
	\end{itemize}
\end{theorem}

On the other hand, after having establish the existence and uniqueness of solutions to the problem \eqref{ras2}-\eqref{ras1} (see, Theorem \ref{unitsolution}), we prove the main homogenization result for this paper. It reads as follows.

\begin{theorem}(Main homogenization result)\label{lem37} \\
	Let $\Phi \in \Delta_{2} \cap \Delta'$ be a Young function such that $\widetilde{\Phi} \in \Delta_{2} \cap \Delta'$ and let $E$ be a \textit{fundamental sequence}.	For each $\epsilon > 0$, let $(u_{\epsilon})_{\epsilon\in E} \in W^{1}_{0}L^{\Phi}_{D_{x}}(Q; L^{\Phi}(\Omega))$ be the unique solution of (\ref{ras2}). Then, as $\epsilon \to 0$, 
	\begin{equation*}
		u_{\epsilon} \rightarrow u_{0} \quad stoch. \;in \; L^{\Phi}(Q\times\Omega)-weak \,2s,  
	\end{equation*}
	and 
	\begin{equation*}
		Du_{\epsilon} \rightarrow Du_{0} + \overline{D}_{\omega}u_{1} + D_{y}u_{2} \quad stoch. \; in \; L^{\Phi}(Q\times\Omega)^{N}-weak \,2s,  
	\end{equation*}
	where $\mathbf{u} = (u_{0}, u_{1}, u_{2}) \in \mathbb{F}_{0}^{1}L^{\Phi}=  W^{1}_{0}L^{\Phi}_{D_{x}}(Q; I_{nv}^{\Phi}(\Omega))\times L^{\Phi}(Q; W^{1}_{\#}L^{\Phi}(\Omega))\times L^{\Phi}(Q\times\Omega; W^{1}_{\#}L^{\Phi}_{per}(Y))$  is the unique solution to the minimization problem 
	\begin{equation*}
		\textup{inf}\left\{ \iiint_{Q\times\Omega\times Y} f(x,\omega,y, \mathbb{D}\mathbf{v})\, dxd\mu dy \, : \, \mathbf{v}=(v_{0}, v_{1}, v_{2}) \in \mathbb{F}_{0}^{1}L^{\Phi} \right\},
	\end{equation*} 
	with $\mathbb{D}\mathbf{v}= Dv_{0} + \overline{D}_{\omega}v_{1} + D_{y}v_{2}$.
\end{theorem}

\subsection{Organization of paper}

The paper is divided into sections each revolving around a specific aspect. Section \ref{sect2} dwells on General notations that we will used throughout this paper. Section \ref{sect3} is devoted to the stochastic two-scale convergence method in Orlicz-Sobolev spaces, and the proofs of two compactness results. Finally, in Section \ref{sect4}, we first sketch out the existence and uniqueness of minimizers for (\ref{ras2}). Next, we establish the preliminary convergence results for \eqref{ras1} and to end, we regularize the integrand $f$ in order to prove the main homogenization result.


\section{General notations}\label{sect2}
Throughout this paper, we refer to this section for the used notations. All the vector spaces are assumed to be real vector spaces, and the scalar functions are assumed real values. For more details, we refer to \cite{adams,allair1,franck,tacha1,tacha3,zand}.
\begin{itemize}
	\item 	The letter $E$ denote any ordinary sequence $E=(\epsilon_{n})$ (integers $n\geq 0$) with $0 \leq \epsilon_{n} \leq 1$ and $\epsilon_{n} \rightarrow 0$ as $n\rightarrow \infty$. Such a sequence will be referred to as a \textit{fundamental sequence}.
	\item $Q$ is a bounded open set of $\mathbb{R}^{N}$ and $Y=[0,1]^{N}$ the unit cube of $\mathbb{R}^{N}$, integer $N>1$.
	\item	$D$ or $D_{x}$ (respectively $D_{y}$) denote the (classical) gradient operator on $Q$ (respectively on $Y$). 
	\item	$\textup{div}$ (respectively $\textup{div}_{y}$) the (classical) divergence operator on $\Omega$ (respectively on $Y$).
	\item $\mathcal{K}(Q) $ is the vector space of continuous functions with compact support in $Q$.
	\item $\mathcal{D}(Q)=\mathcal{C}^{\infty}_{0}(Q)$ is the vector space of smooth functions with compact support in $Q$. 
	\item $\mathcal{C}^{\infty}(Q)$ is the vector space of smooth functions on $Q$. 
	\item $\Phi: [0,\infty)\to[0,\infty)$ is a Young function, i.e. $\Phi$ is continuous, convex, with $\Phi(t)>0$ for $t>0$, $\frac{\Phi(t)}{t}\to 0$ as $t\to 0$, and $\frac{\Phi(t)}{t}\to \infty$ as $t\to \infty$.
	\item $\widetilde{\Phi}$ stands for the complementary (or conjugate) of the Young function $\Phi$, defined by 
	\begin{equation*}
		\widetilde{\Phi}(t)= \sup_{s\geq 0} \big\{st - \Phi(s), \; t\geq 0 \big\}.
	\end{equation*} 
	\item  $\Phi$ is said of class $\Delta_{2}$ at $\infty$ (denoted $\Phi\in \Delta_{2}$) if there are $\alpha>0$ and $t_{0}\geq 0$ such that 
	\begin{equation*}\label{Dela2alpha}
		\Phi(2t) \leq \alpha \Phi(t),\; \textup{for\,all}\, t\geq t_{0}.
	\end{equation*}
	\item  $\Phi$ is of class $\Delta^{\prime}$ at $\infty$ (denoted $\Phi\in \Delta^{\prime}$) if there are $k>0$, $t_{0}\geq 0$ and $s_{0}\geq 0$ such that 
	\begin{equation*}\label{Deltaprime}
		\Phi(ts) \leq k \Phi(t)\Phi(s),\; \textup{for\,all}\, t\geq t_{0} \text{ and } s\geq s_{0}.
	\end{equation*}
	\item $(\Omega, \mathscr{M}, \mu)$ is a measure space with probability measure $\mu$. 
	\item $L^{\Phi}(\Omega)$ is the Orlicz space of functions defined by 
	\begin{equation*}
		L^{\Phi}(\Omega) = \left\{ u: \Omega\to \mathbb{R}\,;\, u\,\textup{is\,measurable},\;   \int_{\Omega}\left(\frac{|u(x)|}{\delta}\right) d\mu < +\infty \text{ for some } \delta>0\right\}.
	\end{equation*}
	\item  $L^{\Phi}(\Omega)$ is a Banach space with respect to the Luxemburg norm 
	\begin{equation*}
		\|u\|_{\Phi,\Omega} = \|u\|_{L^{\Phi}(\Omega)} = \inf \left\{ \delta>0\,:\, \int_{\Omega}\Phi\left(\dfrac{|u(x)|}{\delta}\right)d\mu \leq 1 \right\}.
	\end{equation*}
	\item We assume that $L^{\Phi}(\Omega)$ is separable and reflexive.  Then, the dual of $L^{\Phi}(\Omega)$ is identified with $L^{\widetilde{\Phi}}(\Omega)$.
	\item When $\Omega=Q$, we denote by $W^{1}L^{\Phi}(Q)$ the Orlicz-Sobolev space defined by
	\begin{equation*}
		W^{1}L^{\Phi}(Q) = \left\{ u \in L^{\Phi}(Q)\,;\, \dfrac{\partial u}{\partial x_{i}} \in L^{\Phi}(Q),\, 1\leq i\leq N \right\},
	\end{equation*}
	where derivatives are taken in the distributional sense.
	\item  When $\Phi, \widetilde{\Phi} \in \Delta_{2}$,  $W^{1}L^{\Phi}(Q)$  is a reflexive Banach space  with the norm 
	\begin{equation*}
		\|u\|_{W^{1}L^{\Phi}(Q)} = \|u\|_{L^{\Phi}(Q)} + \sum_{i=1}^{N} \left\|\dfrac{\partial u}{\partial x_{i}} \right\|_{L^{\Phi}(Q)}, \; u \in W^{1}L^{\Phi}(Q).
	\end{equation*}
	\item $W^{1}_{0}L^{\Phi}(Q)$ denotes the closure of $\mathcal{D}(\Omega)$ in $W^{1}L^{\Phi}(Q)$ and the semi-norm 
	\begin{equation*}
		u \longrightarrow \|u\|_{W^{1}_{0}L^{\Phi}(Q)} = \|Du\|_{L^{\Phi}(Q)} =  \sum_{i=1}^{N} \left\|\dfrac{\partial u}{\partial x_{i}} \right\|_{L^{\Phi}(Q)}
	\end{equation*}
	is a norm (of gradient) on $W^{1}_{0}L^{\Phi}(Q)$ equivalent to $\|\cdot\|_{W^{1}L^{\Phi}(Q)}$.
	\item $\mathcal{M}_{Y}(u) = \frac{1}{vol(Y)}\int_{Y}u(y) dy$ is the mean value of an element $u \in L^{1}(Y)$.  
	\item	If $F(\mathbb{R}^{N})$ is a given function space, we denote by $F_{per}(Y)$ the space of functions in $F_{loc}(\mathbb{R}^{N})$ that are $Y$-periodic. 
	\item  As special cases, $\mathcal{D}_{per}(Y)$ denotes the space $\mathcal{C}^{\infty}_{per}(Y)$ while $\mathcal{D}_{\#}(Y)$ stands for the space of those functions in $\mathcal{D}_{per}(Y)$ with mean value zero.
	\item $\mathcal{D}'_{per}(Y)$ stands for the topological dual of $\mathcal{D}_{per}(Y)$ which can be identified with the space of periodic distributions in $\mathcal{D}'(\mathbb{R}^{N})$.
	\item  Also, $W^{1}_{\#}L^{\Phi}_{per}(Y)$ stands for the space of those functions in $W^{1}L^{\Phi}_{per}(Y)$ with mean value zero. It is a same definition for $L^{\Phi}_{\#}(Y)$.
	\item  $\{ T(x), \, x \in \mathbb{R}^{N} \}$ is a fixed $N$-dimensional dynamical system on $\Omega$, with respect to a fixed probability measure $\mu$ on $\Omega$ invariant for $T$. 
	\item	$I^{\Phi}_{nv}(\Omega)$ is the set of all $T$-invariant functions in $L^{\Phi}(\Omega)$.
	\item	$D_{i,\Phi} : L^{\Phi}(\Omega)\rightarrow L^{\Phi}(\Omega)$, $1 \leq i \leq N$, is the $i$-th stochastic derivative operator.
	\item  $D^{\alpha}_{\Phi}= D^{\alpha_{1}}_{1,\Phi}\circ \cdots \circ D^{\alpha_{N}}_{N,\Phi}$, $\alpha=(\alpha_{1}, \cdot, \alpha_{N}) \in \mathbb{N}^{N}$ is the higher order stochastic derivatives. 
	\item	$D_{\omega}=D_{\omega,\Phi} = (D_{1,\Phi}, \cdots, D_{N,\Phi})$ is the stochastic gradient operator on $\Omega$.
	\item $\mathcal{C}^{\infty}(\Omega)$ is the space of functions $f\in L^{\Phi}(\Omega)$, whose each element possesses stochastic derivatives of any order that are bounded.
	\item we denote by $W^{1}L^{\Phi}(\Omega)$ the Orlicz-Sobolev space defined by
	\begin{equation*}
		W^{1}L^{\Phi}(\Omega) = \left\{ u \in L^{\Phi}(\Omega)\,;\, D_{i,\Phi}u \in L^{\Phi}(\Omega),\, 1\leq i\leq N \right\},
	\end{equation*}
	where derivatives are taken in the distributional sense.
	\item  We denote $W^{1}_{\#}L^{\Phi}(\Omega)$ the completion of $\mathcal{C}^{\infty}(\Omega)$ with respect to the seminorm $\|u\|_{\# ,\Phi} = \sum_{i=1}^{N} \|D_{i,\Phi}u\|_{L^{\Phi}(\Omega)}$ and $I_{\Phi} : \mathcal{C}^{\infty}(\Omega) \hookrightarrow W^{1}_{\#}L^{\Phi}(\Omega)$ the canonical mapping.
	\item $\overline{D}_{\omega}=\overline{D}_{\omega,\Phi} = (\overline{D}_{1,\Phi}, \cdots, \overline{D}_{N,\Phi}) : W^{1}_{\#}L^{\Phi}(\Omega) \rightarrow L^{\Phi}(\Omega)^{N}$ denotes the stochastic gradient on $W^{1}_{\#}L^{\Phi}(\Omega)$.
	\item The operator  $\textup{div}_{\omega,\widetilde{\Phi}} : L^{\widetilde{\Phi}}(\Omega)^{N} \rightarrow (W^{1}_{\#}L^{\Phi}(\Omega))'$ is defined by:
	\begin{equation*}
		\langle \textup{div}_{\omega,\widetilde{\Phi}} u \, , \, v \rangle = - \langle u\, , \, \overline{D}_{\omega,\Phi}v \rangle, \quad \forall u \in L^{\widetilde{\Phi}}(\Omega)^{N} \; \textup{and} \; \forall v \in W^{1}_{\#}L^{\Phi}(\Omega).
	\end{equation*}
	\item The Orlicz-Sobolev space $W^{1}L^{\Phi}_{D_{x}}(Q; L^{\Phi}(\Omega))$ is defined by
	\begin{equation}\label{c1orli1S}
		W^{1}L^{\Phi}_{D_{x}}(Q; L^{\Phi}(\Omega)) = \left\{  v \in L^{\Phi}(Q\times\Omega) \, : \, \frac{\partial v}{\partial x_{i}} \in L^{\Phi}(Q\times\Omega), \; 1 \leq i \leq N \right\},
	\end{equation}
	where derivatives are taken in the distributional sense.  This is a reflexive Banach space with the norm:
	\begin{equation*}
		\lVert v \rVert_{W^{1}L^{\Phi}_{D_{x}}(Q; L^{\Phi}(\Omega))} = \lVert v \rVert_{L^{\Phi}(Q\times\Omega)} + \sum_{i=1}^{n} \lVert D_{x_{i}} v \rVert_{L^{\Phi}(Q\times\Omega)}\, ; \; v \in W^{1}L^{\Phi}_{D_{x}}(Q; L^{\Phi}(\Omega)).
	\end{equation*}
	\item $W^{1}_{0}L^{\Phi}_{D_{x}}(Q; L^{\Phi}(\Omega))$ is the set of functions in $W^{1}L^{\Phi}_{D_{x}}(Q; L^{\Phi}(\Omega))$ with zero boundary condition on $Q$. This is a reflexive Banach space with the norm: 
	\begin{equation*}
		\lVert v \rVert_{W^{1}_{0}L^{\Phi}_{D_{x}}(Q; L^{\Phi}(\Omega))} =  \sum_{i=1}^{n} \lVert D_{x_{i}} v \rVert_{L^{\Phi}(Q\times\Omega)} \;\, ; \quad v \in W^{1}_{0}L^{\Phi}_{D_{x}}(Q; L^{\Phi}(\Omega)).
	\end{equation*}
	\item To end this section, we give some examples of Young functions satisfying $\Delta_{2}$ and $\Delta'$ conditions.	The function $t \rightarrow \frac{t^{p}}{p}$, ($p > 1$) is a Young function satisfying $\Delta_{2}\cap \Delta'$-condition. Its  conjugate is a Young function $t \rightarrow \frac{t^{q}}{q}$, where $\frac{1}{p} + \frac{1}{q} = 1$. The function $t \rightarrow t^{p}\ln(1+t)$, ($p \geq 1$) is a Young function satisfying $\Delta_{2}$-condition, while the Young function $t \rightarrow t^{\ln t}$ and $t \rightarrow e^{t^{r}} - 1$, ($r >0$) are not of class $\Delta_{2}$. However, for $r=1$, the Young function $t \rightarrow e^{t} - 1$ satisfy $\Delta'$-condition. 
\end{itemize}


\section{Stochastic two-scale convergence in Orlicz-Sobolev spaces}\label{sect3}

The stochastic-periodic homogenization of functional $F_{\epsilon}$ as defined in (\ref{ras1}), amounts to find a homogenized functional $F$  such that the sequence of minimizers $u_{\epsilon}$ converges to a limit $u$, which is precisely the minimizer of $F$. For this purpose, we have need to extend the \textit{stochastic two-scale convergence} (see, \cite{sango}) to the Orlicz setting.

\subsection{Fundamentals of stochastic-periodic homogenization}

Let us recall that (see, e.g., \cite[Page 185]{tacha1}), putting
\begin{eqnarray*}
	L^{\Phi}(Q\times\Omega\times Y_{per}) = \bigg\{ u \in L^{\Phi}_{loc}(Q\times\Omega\times \mathbb{R}^{N}_{y}) : \, \textup{a.e. \, in} \; (x,\omega) \in Q\times\Omega, \\
	u(x,\omega,\cdot) \,\, \textup{is} \, Y-\textup{periodic} \bigg\},
\end{eqnarray*}
the embedding $L^{\Phi}(Q\times\Omega ; \mathcal{C}_{per}(Y)) \hookrightarrow L^{\Phi}(Q\times\Omega\times Y_{per})$ is continuous with density.
In all this section, we assume that $\Phi$ and $\widetilde{\Phi}$ satisfy the $\Delta_{2}$-condition.
We can now define the stochastic two-scale convergence in the Orlicz space. 
\begin{definition}(stochastic 2-scale convergence in $L^{\Phi}$-spaces)\label{lem36} \\
	Let $\Phi$ be a Young function. A bounded sequence $(u_{\epsilon})$ in $ L^{\Phi}(Q\times\Omega)$ is said to weakly stochastically 2-scale converge to $u_{0} \in  L^{\Phi}(Q\times\Omega\times Y_{per})$, if as $\epsilon \to 0$ we have,  
	\begin{eqnarray}\label{lem35}
		\lim_{\epsilon \to 0} \iint_{Q\times \Omega} u_{\epsilon}(x,\omega) f\left(x, T\left(\frac{x}{\epsilon}\right)\omega,\frac{x}{\epsilon^{2}}\right) dx\, d\mu  \nonumber \\ = \iiint_{Q\times \Omega\times Y} u_{0}(x,\omega, y) f(x,\omega, y) dx\, d\mu dy
	\end{eqnarray}
	for every $f \in \mathcal{C}^{\infty}_{0}(Q)\otimes\mathcal{C}^{\infty}(\Omega)\otimes\mathcal{C}^{\infty}_{per}(Y)$. 
	We express this by writing  $u_{\epsilon} \;\rightarrow\; u_{0}$ stoch. in $ L^{\Phi}(Q\times\Omega)$-weak 2s.
\end{definition}
We recall that $\mathcal{C}^{\infty}_{0}(Q)\otimes\mathcal{C}^{\infty}(\Omega)\otimes\mathcal{C}^{\infty}_{per}(Y)$ is the space of functions of the form, 
\begin{equation}
	f(x, \omega, y) = \sum_{finite} \varphi_{i}(x)\psi_{i}(\omega)g_{i}(y), \quad (x,\omega,y) \in Q\times \Omega\times \mathbb{R}^{N},
\end{equation}
with $\varphi_{i} \in \mathcal{C}^{\infty}_{0}(Q)$, $\psi_{i} \in \mathcal{C}^{\infty}(\Omega)$ and $g_{i} \in \mathcal{C}^{\infty}_{per}(Y)$. Such functions are dense in $\mathcal{C}^{\infty}_{0}(Q) \otimes L^{\widetilde{\Phi}}(\Omega)\otimes\mathcal{C}^{\infty}_{per}(Y)$, since $\mathcal{C}^{\infty}(\Omega)$ is dense in $L^{\widetilde{\Phi}}(\Omega)$ and hence in $\mathcal{K}(Q; L^{\widetilde{\Phi}}(\Omega))\otimes\mathcal{C}^{\infty}_{per}(Y)$, where $\mathcal{K}(Q; L^{\widetilde{\Phi}}(\Omega))$ is the space of continuous functions from $Q$ to $L^{\widetilde{\Phi}}(\Omega)$ with compact support contained in $Q$. As $\mathcal{K}(Q; L^{\widetilde{\Phi}}(\Omega))$ is dense in $L^{\widetilde{\Phi}}(Q\times\Omega)$ and $L^{\widetilde{\Phi}}(Q\times\Omega)\otimes\mathcal{C}^{\infty}_{per}(Y)$ is dense in $L^{\widetilde{\Phi}}(Q\times\Omega ; \mathcal{C}_{per}(Y))$, the uniqueness of the stochastic two-scale limit is ensured.
\begin{remark}(Inspired by \cite{franck,sango})\label{lem16} 
	\begin{enumerate}
		\item[$R_{1})$]Assume that a sequence $(u_{\epsilon})$ in $ L^{\Phi}(Q\times\Omega)$ is weakly stochastically 2-scale convergent to $u_{0} \in  L^{\Phi}(Q\times\Omega\times Y_{per})$. Then since $\mathcal{C}^{\infty}_{0}(Q)\otimes\mathcal{C}^{\infty}(\Omega)\otimes \mathcal{C}^{\infty}_{per}(Y)$ is dense in $\mathcal{C}(Q; L^{\infty}(\Omega ; L^{\infty}_{per}(Y)))$, then (\ref{lem35}) holds for every $f \in \mathcal{C}(Q; L^{\infty}(\Omega ; L^{\infty}_{per}(Y)))$. 
		\item[$R_{2})$]  In the Definition \ref{lem36},  if we take $\Omega = Y \subset \mathbb{R}^{N}$, $d\mu = dy$, and that we consider the $N$-dimensional dynamical system of translation $\{ T(x) : Y\rightarrow Y \; ; \; x \in \mathbb{R}^{N} \}$ defined by $T(x)y = (x+y)\text{mod }\mathbb{Z}^{N}$, we obtain the framework to which reiterated periodic two-scale convergence is defined as in \cite{tacha3}.  
	\end{enumerate}
\end{remark}
Before proceeding with our study, we need to make a comparison between the stochastic two-scale convergence in Orlicz setting and  other existing convergence methods close to it. To achieve this, we must first state these convergence schemes: 
\begin{enumerate}
	\item[(1)]  A sequence $(u_{\epsilon})$ in $ L^{\Phi}(Q)$ is said to weakly 2-scale converge to $v_{0} \in  L^{\Phi}(Q\times Y_{per})$, if as $\epsilon \to 0$ we have,  
	\begin{equation}\label{lem60}
		\lim_{\epsilon \to 0} \int_{Q} u_{\epsilon}(x) f\left(x, \frac{x}{\epsilon}\right) dx = \iint_{Q\times Y} v_{0}(x,y) f(x,y) dx\,dy
	\end{equation}
	for every $f \in \mathcal{C}^{\infty}_{0}(Q)\otimes\mathcal{C}^{\infty}_{per}(Y)$, (see \cite{tacha1}).
	\item[(2)]  A sequence $(u_{\epsilon})$ in $ L^{\Phi}(Q\times\Omega)$ is said to weakly stochastically 2-scale converge in the mean to $v_{0} \in  L^{\Phi}(Q\times\Omega)$, if as $\epsilon \to 0$ we have,  
	\begin{equation}\label{lem61}
		\lim_{\epsilon \to 0} \iint_{Q\times \Omega} u_{\epsilon}(x,\omega) f\left(x, T\left(\frac{x}{\epsilon}\right)\omega\right) dx\, d\mu = \iint_{Q\times \Omega} v_{0}(x,\omega) f(x,\omega) dx\,d\mu
	\end{equation}
	for every $f \in \mathcal{C}^{\infty}_{0}(Q)\otimes\mathcal{C}^{\infty}(\Omega)$, (see \cite{franck}). 
\end{enumerate}
It is very important to note that both of the above definitions (\ref{lem60}) and (\ref{lem61}) imply the boundedness of the sequence $(u_{\epsilon})$ either in $ L^{\Phi}(Q)$ or in $ L^{\Phi}(Q\times\Omega)$, accordingly. With such considerations in mind, we see that the stochastic two-scale convergence method generalizes the above two convergence methods. Indeed, if in (\ref{lem35}) we take $f \in \mathcal{C}^{\infty}_{0}(Q)\otimes\mathcal{C}^{\infty}(\Omega)$, that is $f$ is constant with respect to $y \in \mathbb{R}^{N}$, and next using the density of the latter space in $L^{\widetilde{\Phi}}(Q\times\Omega)$, then (\ref{lem35}) reads as (\ref{lem61}) with $v_{0}(x,\omega) = \int_{Y} u_{0}(x,\omega,y)dy$. If besides we take in $f \in \mathcal{C}^{\infty}_{0}(Q)\otimes\mathcal{C}^{\infty}_{per}(Y)$, that is $f$ not depending upon the random variable $\omega$ and further if we choose $u_{\epsilon}$ not depending on $\omega$, then using the density of $\mathcal{C}^{\infty}_{0}(Q)\otimes\mathcal{C}^{\infty}_{per}(Y)$ in $L^{\widetilde{\Phi}}(Q ; \mathcal{C}_{per}(Y))$, we readily get (\ref{lem60}) with $v_{0}(x,y)=\int_{\Omega}u_{0}(x,\omega,y)d\mu$. 
\par The following lemma  is inspired by Sango-Woukeng paper \cite[Lemma 1 and Remark 2]{sango}.
It is very important for the proof of  second Compactness Theorem (see Subsection \ref{subsec2}).
\begin{lemma}\label{lemc1}
	Let $(u_{\epsilon})$ be a sequence in $L^{\Phi}(Q\times \Omega)$ which stochastically weakly two-scale converges towards $u_{0} \in L^{\Phi}(Q\times\Omega\times Y_{per})$. Let the sequence $(v_{\epsilon})$ be defined by 
	\begin{equation*}
		v_{\epsilon}(x,\omega) = \int_{B_{r}} u_{\epsilon}(x+\epsilon^{2}\rho, \omega) d\rho, \quad (x,\omega) \in Q\times\Omega,
	\end{equation*}
	$B_{r}$ being the open ball of $\mathbb{R}^{N}$ centered at $O$ (origin of $\mathbb{R}^{N}$) and of radius $r>0$. Then, as $\epsilon \to 0$, 
	\begin{equation}\label{lol1}
		v_{\epsilon} \rightarrow v_{0} \; \textup{stoch.\, in}\; L^{\Phi}(Q\times\Omega)-weak\, 2s,
	\end{equation}
	where $v_{0}$ is defined by $v_{0}(x,\omega,y) = \int_{B_{r}} u_{0}(x,\omega, y+\rho) d\rho$, for $(x,\omega,y)\in Q\times\Omega\times\mathbb{R}^{N}$. Moreover, it follows from (\ref{lol1}) that, as $\epsilon \to 0$, 
	\begin{equation*}\label{dedu1}
		\dfrac{1}{|B_{\epsilon^{2}r}|} \int_{B_{\epsilon^{2}r}} u_{\epsilon}(x+\rho, \omega) d\rho \rightarrow \dfrac{1}{|B_{r}|} v_{0} \;\; \textup{stoch.\, in}\; L^{\Phi}(Q\times\Omega)-weak\, 2s.
	\end{equation*}
\end{lemma}
\begin{proof}
	Let $\varphi \in \mathcal{C}^{\infty}_{0}(Q)$, $f \in \mathcal{C}^{\infty}(\Omega)$ and $g \in \mathcal{C}_{per}(Y)$. One has
	\begin{equation*}
		\begin{array}{l}
			\displaystyle	\iint_{Q\times \Omega} \left( \int_{B_{r}} u_{\epsilon}(x+ \epsilon^{2}\rho , \omega) d\rho \right) \varphi(x) f\left(T\left(\frac{x}{\epsilon}\right)\omega \right) g\left(\frac{x}{\epsilon^{2}}\right) dx d\mu  \\
			= \displaystyle \int_{B_{r}}	 \left( \iint_{Q\times \Omega} u_{\epsilon}(x+ \epsilon^{2}\rho , \omega) \varphi(x) f\left(T\left(\frac{x}{\epsilon}\right)\omega \right) g\left(\frac{x}{\epsilon^{2}}\right) dx d\mu \right) d\rho. 
		\end{array}
	\end{equation*}
	In view of the Lebesgue Dominated Convergence Theorem, (\ref{lol1}) will be checked as soon as we show that for each fixed $\rho \in \mathbb{R}^{N}$ one has, as $E\ni \epsilon \to 0$,
	\begin{equation*}
		\begin{array}{l}
			\displaystyle	\iint_{Q\times \Omega}  u_{\epsilon}(x+ \epsilon^{2}\rho , \omega) \varphi(x) f\left(T\left(\frac{x}{\epsilon}\right)\omega \right) g\left(\frac{x}{\epsilon^{2}}\right) dx d\mu  \\
			\rightarrow	 \displaystyle  \iint_{Q\times \Omega}\int_{Y} (\tau_{-\rho}u_{0})(x , \omega, y) \varphi(x) f\left(\omega \right) g\left(y\right) dydx d\mu, 
		\end{array}
	\end{equation*}
	where the function $\tau_{-\rho}u_{0}$ is defined by $(\tau_{-\rho}u_{0})(x,\omega,y) = u_{0}(x,\omega, y+\rho)$. So let $a \in \mathbb{R}^{N}$ and let $\varphi$, $f$ and $g$ be as above; then 
	\begin{equation*}
		\begin{array}{l}
			\displaystyle	\iint_{Q\times \Omega}  u_{\epsilon}(x- \epsilon^{2}a , \omega) \varphi(x) f\left(T\left(\frac{x}{\epsilon}\right)\omega \right) g\left(\frac{x}{\epsilon^{2}}\right) dx d\mu  \\
			=	\displaystyle  \iint_{(Q-\epsilon^{2}a)\times \Omega} u_{\epsilon}(x,\omega) \varphi(x+\epsilon^{2}a) f\left(T\left(\frac{x}{\epsilon}+\epsilon a\right)\omega \right) g\left(\frac{x}{\epsilon^{2}}+a\right) dx d\mu \\
			= \displaystyle \iint_{Q\times \Omega} u_{\epsilon}(x,\omega) \varphi(x+\epsilon^{2}a) f\left(T\left(\frac{x}{\epsilon}+\epsilon a\right)\omega \right) g\left(\frac{x}{\epsilon^{2}}+a\right) dx d\mu  \\
			\quad	- \displaystyle \iint_{(Q\backslash Q-\epsilon^{2}a)\times \Omega} u_{\epsilon}(x,\omega) \varphi(x+\epsilon^{2}a) f\left(T\left(\frac{x}{\epsilon}+\epsilon a\right)\omega \right) g\left(\frac{x}{\epsilon^{2}}+a\right) dx d\mu \\
			\quad 	+ \displaystyle \iint_{(Q-\epsilon^{2}a)\backslash Q\times \Omega} u_{\epsilon}(x,\omega) \varphi(x+\epsilon^{2}a) f\left(T\left(\frac{x}{\epsilon}+\epsilon a\right)\omega \right) g\left(\frac{x}{\epsilon^{2}}+a\right) dx d\mu  \\
			= (I) - (II) + (III).
		\end{array}
	\end{equation*}
	As for $(I)$ we have 
	\begin{equation*}
		\begin{array}{rcl}
			(I) & = & \displaystyle \iint_{Q\times \Omega} u_{\epsilon}(x,\omega) \varphi(x) f\left(T\left(\frac{x}{\epsilon}+\epsilon a\right)\omega \right) (\tau_{-a}g)\left(\frac{x}{\epsilon^{2}}\right) dx d\mu  \\
			&  & + \displaystyle \iint_{Q\times \Omega} u_{\epsilon}(x,\omega) [\varphi(x+\epsilon^{2}a)-\varphi(x)] f\left(T\left(\frac{x}{\epsilon}+\epsilon a\right)\omega \right) (\tau_{-a}g)\left(\frac{x}{\epsilon^{2}}\right) dx d\mu \\
			& = & (I_{1}) + (I_{2}).
		\end{array}
	\end{equation*}
	But, 
	\begin{equation*}
		\begin{array}{rcl}
			(I_{1}) & = & \displaystyle \iint_{Q\times \Omega} u_{\epsilon}(x,\omega) \varphi(x) f\left(T\left(\frac{x}{\epsilon} \right)\omega \right) (\tau_{-a}g)\left(\frac{x}{\epsilon^{2}}\right) dx d\mu  \\
			&  & + \displaystyle \iint_{Q\times \Omega} u_{\epsilon}(x,\omega)\varphi(x)(\tau_{-a}g)\left(\frac{x}{\epsilon^{2}}\right) \left[f\left(T\left(\frac{x}{\epsilon}+\epsilon a\right)\omega \right) - f\left(T\left(\frac{x}{\epsilon}\right)\omega \right)  \right]   dx d\mu \\
			& = & (I'_{1}) + (I'_{2}).
		\end{array}
	\end{equation*}
	As for $(I'_{1})$, we have 
	\begin{equation*}
		(I'_{1}) \rightarrow \iint_{Q\times \Omega} \int_{Y} u_{0}(x,\omega,y) \varphi(x) f(\omega) (\tau_{-a}g)(y) dy dx d\mu, \quad \textup{as} \; E \ni \epsilon \to 0.
	\end{equation*}
	But, 
	\begin{equation*}
		\begin{array}{rcl}
			\displaystyle \int_{Y} u_{0}(x,\omega,y)  (\tau_{-a}g)(y) dy & = & \mathcal{M}_{Y}\left( u_{0}(x,\omega,\cdot)  (\tau_{-a}g) \right)  \\ 
			& = &  \mathcal{M}_{Y}\left(\tau_{-a}[\tau_{a} u_{0}(x,\omega,\cdot)g] \right)   \\
			& = &  \mathcal{M}_{Y}\left(\tau_{a} u_{0}(x,\omega,\cdot)g \right)  \\
			& = & \displaystyle \int_{Y} \tau_{a} u_{0}(x,\omega,y)g(y) dy. 
		\end{array}
	\end{equation*}
	For $(I'_{2})$ we have 
	\begin{equation*}
		\begin{array}{lll}
			|(I'_{2})|  &\leq&  2 \|u_{\epsilon}\|_{\Phi,Q\times\Omega}\times \|\varphi\|_{\infty} \times\|g\|_{\infty}\times \\
			& & \left(  \inf\left\{ \delta>0 :  \, \displaystyle \iint_{Q\times \Omega} \widetilde{\Phi}\left( \dfrac{ \left| f\left(T\left(\frac{x}{\epsilon}+\epsilon a\right)\omega \right) - f\left(T\left(\frac{x}{\epsilon}\right)\omega \right) \right| }{\delta}  \right) dx d\mu \leq 1  \right\} \right).
		\end{array}
	\end{equation*}	
	But, 
	\begin{equation*}
		\begin{array}{l}
			\inf\left\{ \delta>0 :  \, \displaystyle \iint_{Q\times \Omega} \widetilde{\Phi}\left( \dfrac{ \left| f\left(T\left(\frac{x}{\epsilon}+\epsilon a\right)\omega \right) - f\left(T\left(\frac{x}{\epsilon}\right)\omega \right) \right| }{\delta}  \right) dx d\mu \leq 1  \right\}  \\
			=  \inf\left\{ \delta>0 :  \, \displaystyle \iint_{Q\times \Omega} \widetilde{\Phi}\left( \dfrac{ \left| \left(U\left(\frac{x}{\epsilon}+\epsilon a\right)f \right)(\omega) - \left(U\left(\frac{x}{\epsilon}\right)f \right)(\omega) \right| }{\delta}  \right) dx d\mu \leq 1  \right\}. 
		\end{array}
	\end{equation*}
	Since the group $U(x)$ is strongly continuous in $L^{\widetilde{\Phi}}(\Omega)$ (see, \cite[Proposition 3]{franck}), we get immediately that 
	\begin{equation*}
		\begin{array}{r}
			\inf\left\{ \delta>0 :  \, \displaystyle \iint_{Q\times \Omega} \widetilde{\Phi}\left( \dfrac{ \left| f\left(T\left(\frac{x}{\epsilon}+\epsilon a\right)\omega \right) - f\left(T\left(\frac{x}{\epsilon}\right)\omega \right) \right| }{\delta}  \right) dx d\mu \leq 1  \right\}  \rightarrow 0, \quad \\ \textup{as} \; \epsilon \to 0.
		\end{array}
	\end{equation*}
	Thus $(I'_{2}) \rightarrow 0$ as $\epsilon \to 0$. Finally, since the sequence $(u_{\epsilon})$ is bounded in $L^{\Phi}(Q\times \Omega)$, this sequence is uniformly integrable in $L^{1}(Q\times \Omega)$. Therefore, the inequality 
	\begin{equation*}
		\begin{array}{l}
			\displaystyle	\iint_{((Q-\epsilon^{2}a)\Delta Q)\times \Omega} |u_{\epsilon}(x,\omega)| |\varphi(x+\epsilon^{2}a)| \left| f\left(T\left(\frac{x}{\epsilon}+\epsilon a\right)\omega \right) \right| \left| g\left(\frac{x}{\epsilon^{2}}+a\right)\right| dx d\mu  \\
			\leq  \|\varphi\|_{\infty} \|f\|_{\infty} \|g\|_{\infty} \displaystyle\iint_{((Q-\epsilon^{2}a)\Delta Q)\times \Omega} |u_{\epsilon}(x,\mu)| dx d\mu,  
		\end{array}
	\end{equation*}
	yields that $(II)$ and $(III)$ goes towards 0 as $E \ni \epsilon \to 0$; here, the symbol $\Delta$ between the sets $(Q-\epsilon^{2}a)$ and $Q$ denotes the symmetric difference between these two sets. The proof is complete. 	
\end{proof}

\subsection{Compactness theorem 1: case of Orlicz's spaces}
The proof of the first compactness theorem requires the following lemma.
\begin{lemma}\label{lem2}
	Let $f \in \mathcal{C}^{\infty}_{0}(Q)\otimes\mathcal{C}^{\infty}(\Omega)\otimes \mathcal{C}_{per}(Y)$. Then, for any $\epsilon >0$, the function $f_{T}^{\epsilon} : (x, \omega, y) \rightarrow f\left(x, T\left(\frac{x}{\epsilon}\right)\omega, \frac{x}{\epsilon^{2}} \right)$ is such that 
	\begin{equation}\label{lem1}
		\lim_{\epsilon\to 0} \|f_{T}^{\epsilon}\|_{L^{\Phi}(Q\times \Omega)}  =  \|f\|_{L^{\Phi}(Q\times \Omega\times Y)}.
	\end{equation}
\end{lemma}
\begin{proof}
	We just need to show (\ref{lem1}) for each $f \in \mathcal{C}^{\infty}_{0}(Q)\otimes\mathcal{C}^{\infty}(\Omega)\otimes \mathcal{C}_{per}(Y)$. For that, it is sufficient to do it for $f$ under the form $f(x, \omega,y) = \phi(x)\psi(\omega)g(y)$ with $\phi \in \mathcal{C}^{\infty}_{0}(Q)$, $\psi \in \mathcal{C}^{\infty}(\Omega)$ and $g \in \mathcal{C}_{per}(Y)$. \\
	Given the fact that the measure $\mu$ is invariant under the dynamical system $T$ and by \cite[Proposition 4.3]{tacha1}, 
	\begin{equation*}
		\begin{array}{rcl}
			\displaystyle 	\lim_{\epsilon\to 0} \| f_{T}^{\epsilon} \|_{\Phi, Q\times\Omega}  &= & \displaystyle  \lim_{\epsilon\to 0} \inf \left\{k >0 \; ; \;\, \iint_{Q\times \Omega} \Phi\left(\dfrac{\phi(x)\psi(T\left(\frac{x}{\epsilon}\right)\omega)g\left(\frac{x}{\epsilon^{2}}\right) }{k} \right)dxd\mu \leq 1 \right\} \\
			&= & \displaystyle  \lim_{\epsilon\to 0} \inf \left\{k >0 \; ; \;\, \iint_{Q\times \Omega} \Phi\left(\dfrac{\phi(x)g\left(\frac{x}{\epsilon^{2}}\right)\psi(\omega) }{k} \right)dxd\mu_{T} \leq 1 \right\} \\
			& =& \inf \left\{k >0 \; ; \;\, \displaystyle \lim_{\epsilon\to 0}  \displaystyle \iint_{Q\times \Omega} \Phi\left(\dfrac{\phi(x)g\left(\frac{x}{\epsilon^{2}}\right)\psi(\omega) }{k} \right)dxd\mu \leq 1 \right\}  \\
			& = & \inf \left\{k >0 \; ; \;\,   \displaystyle \iiint_{Q\times \Omega\times Y} \Phi\left(\dfrac{\phi(x)g(y)\psi(\omega) }{k} \right)dxd\mu dy \leq 1 \right\}  \\
			& = & \| f \|_{\Phi, Q\times\Omega\times Y}.
		\end{array}
	\end{equation*}
\end{proof}
With this lemma, we are in position to prove our first compactness result whose we recall here for the reader's convenience.

\noindent \textbf{Theorem \ref{lem4}.}(compactness 1) \\ 
\textit{	Let $\Phi \in \Delta_{2}$ be a Young function.	Any bounded sequence $(u_{\epsilon})_{\epsilon\in E}$  in $L^{\Phi}(Q\times\Omega)$ admits a subsequence which is weakly stochastically 2-scale convergent in $ L^{\Phi}(Q\times\Omega)$. } 

\begin{proof}
	Let $\Phi \in \Delta_{2}$ be a Young function and $\widetilde{\Phi} \in \Delta_{2}$ its conjugate, $(u_{\epsilon})_{\epsilon\in E}$ a bounded sequence  in $L^{\Phi}(Q\times\Omega)$. \\
	In order to use \cite[Proposition 3.2]{gabri}, we get $Y = L^{\widetilde{\Phi}}(Q\times\Omega\times Y_{per})$ and $X = \mathcal{C}^{\infty}_{0}(Q)\otimes\mathcal{C}^{\infty}(\Omega)\otimes \mathcal{C}_{per}(Y)$.
	For each $f \in \mathcal{C}^{\infty}_{0}(Q)\otimes\mathcal{C}^{\infty}(\Omega)\otimes \mathcal{C}_{per}(Y)$ and $\epsilon > 0$, let us define the functional $\Gamma_{\epsilon}$ by 
	\begin{equation}
		\Gamma_{\epsilon}(f) = \iint_{Q\times \Omega} u_{\epsilon}(x,\omega) f\left(x, T\left(\frac{x}{\epsilon}\right)\omega, \frac{x}{\epsilon^{2}} \right) dx\, d\mu.
	\end{equation}
	Then by Holder's inequality and the Lemma \ref{lem2}, we have
	\begin{equation*}\label{fra1}
		\limsup_{\epsilon} |\Gamma_{\epsilon}(f)| \leq 2c \, \|f\|_{L^{\widetilde{\Phi}}(Q\times\Omega\times Y)},
	\end{equation*}
	where $\displaystyle c = \sup_{\epsilon} \|u_{\epsilon}\|_{L^{\Phi}(Q\times\Omega)}$. \\
	Hence, we deduce from \cite[Proposition 3.2]{gabri} the existence of a subsequence $E'$ of $E$ and of a unique $u_{0} \in Y' = L^{\Phi}(Q\times\Omega\times Y_{per})$ such that 
	\begin{eqnarray*}
		\lim_{\epsilon \to 0} \iint_{Q\times \Omega} u_{\epsilon}(x,\omega) f\left(x, T\left(\frac{x}{\epsilon}\right)\omega, \frac{x}{\epsilon^{2}} \right) dx\, d\mu dy  \\ 
		= \iiint_{Q\times \Omega\times Y} u_{0}(x,\omega,y) f(x,\omega,y) dx\, d\mu dy
	\end{eqnarray*}
	for all $f \in \mathcal{C}^{\infty}_{0}(Q)\otimes\mathcal{C}^{\infty}(\Omega)\otimes \mathcal{C}_{per}(Y)$.
\end{proof}

\subsection{Compactness theorem 2: case of Orlicz-Sobolev spaces}\label{subsec2}
To prove the second compactness result, we will need the following lemma whose proof can be found in \cite[Lemma 11]{franck}.
\begin{lemma}\cite{franck}\label{lem5}
	Let $\Phi \in \Delta_{2}$ be a Young function and $\widetilde{\Phi}$ its conjugate.	Let $v \in L^{\Phi}(\Omega)^{N}$ satisfy 
	\begin{equation*}
		\int_{\Omega} v\cdot g\, d\mu = 0, \quad \textup{for \; all} \;\, g \in  \mathcal{V}^{\widetilde{\Phi}}_{\textup{div} }= \{ f \in \mathcal{C}^{\infty}(\Omega)^{N} : \textup{div}_{\omega,\widetilde{\Phi}}f=0 \}.
	\end{equation*}
	Then, there exists $u \in W^{1}_{\#}L^{\Phi}(\Omega)$ such that 
	\begin{equation*}
		v = \overline{D}_{\omega,\Phi}u.
	\end{equation*} 
\end{lemma}
We now recall and prove the second compactness result as follows.

\noindent \textbf{Theorem \ref{lem29}.}(Compactness 2)\\  
\textit{	Let $\Phi \in \Delta_{2}$ be a Young function and $\widetilde{\Phi} \in \Delta_{2}$ its conjugate.
	Assume $(u_{\epsilon})_{\epsilon\in E}$ is a sequence in $W^{1}L^{\Phi}_{D_{x}}(Q; L^{\Phi}(\Omega))$ such that: } 
\begin{itemize}
	\item[(i)] \textit{ $(u_{\epsilon})_{\epsilon\in E}$ is bounded in $L^{\Phi}\left(Q\times\Omega\right)$ and $(D_{x}u_{\epsilon})_{\epsilon\in E}$ is bounded in $L^{\Phi}\left(Q\times\Omega\right)^{N}$.}
\end{itemize}
\textit{	Then there exist $u_{0} \in W^{1}L^{\Phi}_{D_{x}}(Q; I_{nv}^{\Phi}(\Omega))$, $u_{1} \in L^{1}\left(Q; W^{1}_{\#}L^{\Phi}(\Omega)\right)$ with $\overline{D}_{\omega}u_{1} \in L^{\Phi}(Q\times\Omega)^{N}$, $u_{2}\in L^{1}(Q\times\Omega; W^{1}_{\#}L^{\Phi}_{per}(Y))$ with $D_{y}u_{2} \in L^{\Phi}(Q\times\Omega\times Y_{per})^{N}$ and a subsequence $E'$ from $E$ such that, as $E' \ni \epsilon \rightarrow 0$, }
\begin{itemize}
	\item[(ii)] \textit{ $u_{\epsilon} \rightarrow u_{0}$ stoch. in $L^{\Phi}(Q\times \Omega)$-weak 2s;}
	\item[(iii)] \textit{ $D_{x}u_{\epsilon} \rightarrow D_{x}u_{0} + \overline{D}_{\omega}u_{1} + D_{y}u_{2}$ stoch. in $L^{\Phi}(Q\times \Omega)^{N}$-weak 2s,  with $u_{1} \in L^{\Phi}(Q; W^{1}_{\#}L^{\Phi}(\Omega))$, $u_{2}\in L^{\Phi}(Q\times\Omega; W^{1}_{\#}L^{\Phi}_{per}(Y))$  when $\widetilde{\Phi} \in \Delta'$. } 
\end{itemize}

\begin{proof}
	Let us first note that by hypothesis $(i)$, Theorem \ref{lem4} implies the existence of a subsequence $E'$ from $E$, a function $u_{0} \in L^{\Phi}(Q\times \Omega\times Y_{per})$ and a vector function $\mathbf{v} = (v_{i})_{1\leq i \leq N} \in L^{\Phi}(Q\times \Omega\times Y_{per})^{N}$ such that, as $E' \ni \epsilon \rightarrow 0$, we have  $u_{\epsilon} \rightarrow u_{0}$ stoch. in $L^{\Phi}(Q\times \Omega)$-weak 2s and 
	\begin{equation}\label{m2}
		D_{x}u_{\epsilon} \rightarrow \mathbf{v} \;\;\; \textup{stoch.\;\, in} \; L^{\Phi}(Q\times \Omega)^{N}-weak\, 2s.
	\end{equation} 
	We must check that:
	\begin{enumerate}
		\item $(\star)$ $u_{0}$ does not depend upon $y$, that is $D_{y}u_{0}=0$ ; \\ $(\star\star)$ $u_{0}(x,\cdot) \in I_{nv}^{\Phi}(\Omega)$, that is $D_{\omega}u_{0}(x,\cdot) = 0$ or equivalently, \\ $\int_{\Omega} u_{0}(x,\cdot)D_{i,\Phi}\varphi \, d\mu = 0$ for all $\varphi \in \mathcal{C}^{\infty}(\Omega)$, (see \cite[Page 299]{franck}) and,  \\ $(\star\star\star)$ $u_{0} \in W^{1}L^{\Phi}_{D_{x}}(Q; I_{nv}^{\Phi}(\Omega))$ ;
		\item  there exist two functions $u_{1} \in L^{1}\left(Q; W^{1}_{\#}L^{\Phi}(\Omega)\right)$ with $\overline{D}_{\omega}u_{1} \in L^{\Phi}(Q\times\Omega)^{N}$, $u_{2}\in L^{1}(Q\times\Omega; W^{1}_{\#}L^{\Phi}_{per}(Y))$ with $\overline{D}_{y}u_{2} \in L^{\Phi}(Q\times\Omega\times Y_{per})^{N}$ such that $\mathbf{v} = D_{x}u_{0} + \overline{D}_{\omega}u_{1} + D_{y}u_{2}$. \\
		The fact that $u_{1} \in L^{\Phi}(Q; W^{1}_{\#}L^{\Phi}(\Omega))$, $u_{2}\in L^{\Phi}(Q\times\Omega; W^{1}_{\#}L^{\Phi}_{per}(Y))$  if $\widetilde{\Phi} \in \Delta'$   will follow from \cite[Theorem 12]{franck} and \cite[Remark 2]{tacha2}.
	\end{enumerate}
	Let us first check (1) : $(\star)$ Let $\mathbf{\Phi}_{\epsilon}(x,\omega) = \epsilon^{2} \varphi(x)f\left(T\left(\frac{x}{\epsilon}\right)\omega \right)g\left(\frac{x}{\epsilon^{2}}\right)$ for $(x,\omega) \in Q\times \Omega$ where $\varphi \in \mathcal{C}^{\infty}_{0}(Q)$, $f \in \mathcal{C}^{\infty}(\Omega)$ and $g\in \mathcal{D}_{per}(Y)$. Then 
	\begin{equation*}
		\begin{array}{rcl}
			\displaystyle 	\iint_{Q\times \Omega} \dfrac{\partial u_{\epsilon}}{\partial x_{i}}\mathbf{\Phi}_{\epsilon} \, dxd\mu & = & - \displaystyle \iint_{Q\times \Omega} \epsilon^{2}\, u_{\epsilon} f^{\epsilon} g^{\epsilon}\dfrac{\partial \varphi}{\partial x_{i}}dxd\mu
			- \displaystyle\iint_{Q\times \Omega} u_{\epsilon} \varphi f^{\epsilon} (D_{y_{i}}g)^{\epsilon} dxd\mu \\
			& & 
			- \displaystyle\iint_{Q\times \Omega} \epsilon\, u_{\epsilon}\varphi (D_{i,\Phi}f^{\epsilon})dxd\mu, 
		\end{array} 
	\end{equation*}
	where $D_{y_{i}}g = \frac{\partial g}{\partial y_{i}}$, $f^{\epsilon}(x,\omega) = f\left(T\left(\frac{x}{\epsilon}\right)\omega \right)$ and $g^{\epsilon}(x)= g\left(\frac{x}{\epsilon^{2}}\right)$. Letting $E' \ni \epsilon \rightarrow 0$, we get 
	\begin{equation*}
		\iiint_{Q\times \Omega\times Y}  u_{0}\varphi(D_{y_{i}}g)f\, dxd\mu dy = 0,
	\end{equation*} 
	hence $\int_{Y} u_{0}(x,\omega,\cdot) D_{y_{i}}g \, dy = 0$ for all $g\in \mathcal{D}_{per}(Y)$ and all $1\leq i \leq N$, which means that $u_{0}$ does not depend on $y$.\\
	$(\star\star)$ Let $\mathbf{\Psi}_{\epsilon}(x,\omega) = \epsilon \varphi(x)f\left(T\left(\frac{x}{\epsilon}\right)\omega \right)$ for $(x,\omega) \in Q\times \Omega$ where $\varphi \in \mathcal{C}^{\infty}_{0}(Q)$ and $f \in \mathcal{C}^{\infty}(\Omega)$. Then, 
	\begin{equation*}
		\iint_{Q\times \Omega} \dfrac{\partial u_{\epsilon}}{\partial x_{i}}\mathbf{\Psi}_{\epsilon} \, dxd\mu = - \iint_{Q\times \Omega} \epsilon\, u_{\epsilon} f^{\epsilon}\dfrac{\partial \varphi}{\partial x_{i}}dxd\mu - \iint_{Q\times \Omega}  u_{\epsilon}\varphi(D_{i,\Phi}f^{\epsilon})dxd\mu, 
	\end{equation*}
	where $f^{\epsilon}(x,\omega) = f\left(T\left(\frac{x}{\epsilon}\right)\omega \right)$. Letting $E' \ni \epsilon \rightarrow 0$, we get 
	\begin{equation*}
		\iint_{Q\times \Omega}  u_{0}\varphi(D_{i,\Phi}f)dxd\mu = 0.
	\end{equation*} 
	Hence, $\displaystyle \int_{\Omega} u_{0}(x,\cdot)D_{i,\Phi}f d\mu = 0 $ for all $1\leq i \leq N$ and all $f \in \mathcal{C}^{\infty}(\Omega)$, which is equivalent to say that $u_{0}(x,\cdot) \in I_{nv}^{\Phi}(\Omega)$ for a.e. $x \in Q$. \\
	$(\star\star\star)$ Hypothesis (i) implies that the sequence $(u_{\epsilon})_{\epsilon\in E'}$ is bounded in $W^{1}L^{\Phi}_{D_{x}}(Q ; L^{\Phi}(\Omega))$, which yields the existence of a subsequence of $E'$ not relabeled and of a function $u \in W^{1}L^{\Phi}_{D_{x}}(Q ; L^{\Phi}(\Omega))$ such that $u_{\epsilon} \rightarrow u$ in $W^{1}L^{\Phi}_{D_{x}}(Q ; L^{\Phi}(\Omega))$-weak as $E' \ni \epsilon \rightarrow 0$. \\ In particular $\displaystyle \int_{\Omega} u_{\epsilon}(\cdot,\omega)\psi(\omega)d\mu \rightarrow \int_{\Omega} u(\cdot,\omega)\psi(\omega)d\mu$ in $L^{1}(Q)$-weak for all $\psi \in I_{nv}^{\widetilde{\Phi}}(\Omega)$. \\Therefore, using \cite[Lemma 3.6]{Bourgeat}, we get at once $u_{0} \in W^{1}L^{\Phi}_{D_{x}}(Q ; L^{\Phi}(\Omega))$ so that $u_{0} \in W^{1}L^{\Phi}_{D_{x}}(Q ; I_{nv}^{\Phi}(\Omega))$. \\
	We still have to check (2). Let us start by showing first the existence of $u_{2}\in L^{\Phi}(Q\times\Omega; W^{1}_{\#}L^{\Phi}_{per}(Y))$. For this purpose, let $r>0$ be fixed. Let $B_{\epsilon^{2}r}$ denote the open ball in $\mathbb{R}^{N}$ centered at the origin and of radius $\epsilon^{2}r$. By the equalities
	\begin{equation*}
		\begin{array}{l}
			\dfrac{1}{\epsilon^{2}} \left( u_{\epsilon}(x,\omega) - \dfrac{1}{|B_{\epsilon^{2}r}|}\displaystyle\int_{B_{\epsilon^{2}r}} u_{\epsilon}(x+\rho, \omega) d\rho \right) \\ =  \dfrac{1}{\epsilon^{2}}  \dfrac{1}{|B_{\epsilon^{2}r}|} \displaystyle\int_{B_{\epsilon^{2}r}} (u_{\epsilon}(x,\omega) - u_{\epsilon}(x+\rho, \omega) ) d\rho \\
			=  \dfrac{1}{\epsilon^{2}}  \dfrac{1}{|B_{r}|}\displaystyle \int_{B_{r}} (u_{\epsilon}(x,\omega) - u_{\epsilon}(x+ \epsilon^{2}\rho, \omega) ) d\rho \\
			=  - \dfrac{1}{|B_{r}|} \displaystyle\int_{B_{r}} d\rho \int_{0}^{1} Du_{\epsilon}(x+ t\epsilon^{2}\rho, \omega) \cdot \rho \; dt,
		\end{array}
	\end{equation*} 
	we deduce from the boundedness of $(u_{\epsilon})_{\epsilon\in E'}$ in $W^{1}L^{\Phi}_{D_{x}}(Q; L^{\Phi}(\Omega))$ that the sequence $(z^{r}_{\epsilon})_{\epsilon\in E'}$ defined by 
	\begin{equation*}
		z^{r}_{\epsilon}(x,\omega) = \dfrac{1}{\epsilon^{2}} \left( u_{\epsilon}(x,\omega) - \dfrac{1}{|B_{\epsilon^{2}r}|}\int_{B_{\epsilon^{2}r}} u_{\epsilon}(x+\rho, \omega) d\rho \right), \quad (x,\omega) \in Q\times\Omega,\; \epsilon\in E',
	\end{equation*}
	is bounded in $L^{\Phi}(Q\times\Omega)$. Thus, it results once more from Theorem \ref{lem4} the existence of a subsequence from $E'$ not relabeled and of a function $z_{r}$ in $L^{\Phi}(Q\times\Omega\times Y_{per})$ such that, as $E' \ni \epsilon \to 0$,
	\begin{equation}\label{m1}
		z_{\epsilon}^{r} \rightarrow z_{r} \; \textup{stoch. \, in} \; L^{\Phi}(Q\times\Omega)\textup{-weak\, 2s}.
	\end{equation}
	Now, for $\varphi\in \mathcal{C}^{\infty}_{0}(Q)$, $f\in \mathcal{D}(\Omega)$ and $g\in \mathcal{D}_{per}(Y)$ we have, 
	\begin{equation}\label{lemc2}
		\begin{array}{l}
			\displaystyle	\iint_{Q\times \Omega}  \left( \dfrac{\partial u_{\epsilon}}{\partial x_{i}} (x,\omega) - \dfrac{1}{|B_{\epsilon^{2}r}|}\int_{B_{\epsilon^{2}r}}\dfrac{\partial u_{\epsilon}}{\partial x_{i}}(x+\rho, \omega) d\rho \right) \varphi(x) f\left(T\left(\frac{x}{\epsilon}\right)\omega \right) g\left(\frac{x}{\epsilon^{2}}\right) dx d\mu  \\
			= - \displaystyle\iint_{Q\times \Omega} \epsilon^{2} z^{r}_{\epsilon}(x,\omega) f\left(T\left(\frac{x}{\epsilon}\right)\omega \right) g\left(\frac{x}{\epsilon^{2}}\right) \dfrac{\partial \varphi}{\partial x_{i}}(x) dx d\mu \\
			\;\;\;	-  \displaystyle\iint_{Q\times \Omega} \epsilon z^{r}_{\epsilon}(x,\omega) \varphi(x)  g\left(\frac{x}{\epsilon^{2}}\right) (D_{i,\Phi} f)\left(T\left(\frac{x}{\epsilon}\right)\omega \right) dx d\mu \\
			\;\;\;	- \displaystyle \iint_{Q\times \Omega}  z^{r}_{\epsilon}(x,\omega) \varphi(x) f\left(T\left(\frac{x}{\epsilon}\right)\omega \right) \dfrac{\partial g}{\partial x_{i}}\left(\frac{x}{\epsilon^{2}}\right)  dx d\mu. 
		\end{array}
	\end{equation}
	Passing to the limit in (\ref{lemc2}), as $E' \ni \epsilon \to 0$, using (\ref{m2}), (\ref{m1}) and Lemma \ref{lemc1}, one gets 
	\begin{equation*}
		\begin{array}{l}
			\displaystyle \iiint_{Q\times \Omega\times Y}  \left( v_{i}(x,\omega,y) - \dfrac{1}{|B_{r}|}\int_{B_{r}}v_{i}(x, \omega, y+\rho) d\rho \right) \varphi(x) f(\omega) g(y) dx d\mu dy  \\
			= - \displaystyle \iiint_{Q\times \Omega\times Y}  z_{r}(x,\omega,y) \varphi(x) f(\omega) \dfrac{\partial g}{\partial y_{i}}(y) dx d\mu dy.
		\end{array}
	\end{equation*}
	Since $\varphi$, $f$ and $g$ are arbitrary, we get that
	\begin{equation*}
		\dfrac{\partial z_{r}}{\partial y_{i}}(x,\omega,y) =  v_{i}(x,\omega,y) - \dfrac{1}{|B_{r}|}\int_{B_{r}}v_{i}(x, \omega, y+\rho) d\rho,
	\end{equation*}
	a.e. in $x\in Q$, $\omega \in \Omega$ and $y \in \mathbb{R}^{N}$. \\
	Set $f_{r}(x,\omega,y) = z_{r}(x,\omega,y) - \mathcal{M}(z_{r}(x,\omega,\cdot))$ for a.e. in $(x,\omega,y)\in Q\times \Omega \times \mathbb{R}^{N}$, where $\mathcal{M}(z_{r}(x,\omega,\cdot)) = \int_{Y} z_{r}(x,\omega,y) dy$. Then $\mathcal{M}(f_{r}(x,\omega,\cdot)) = 0$ and moreover $D_{y}f_{r}(x,\omega,\cdot) = D_{y}z_{r}(x,\omega,\cdot)$ for a.e. $(x,\omega) \in Q\times \Omega$, so that $f_{r}(x,\omega,\cdot) \in  L^{\Phi}_{\#}(Y)$ with $\dfrac{\partial f_{r}}{\partial y_{i}} (x,\omega,\cdot) \in  L^{\Phi}_{per}(Y)$ for all $1\leq i \leq N$, that is $f_{r}(x,\omega,\cdot) \in  W^{1}_{\#}L^{\Phi}_{per}(Y)$.
	\\	On the other hand, $D_{y}f_{r} \in L^{\Phi}(Q\times\Omega\times Y_{per})^{N} \subset  L^{1}\left(Q\times\Omega; L^{\Phi}_{per}(Y)\right)$, (see, e.g. \cite[Page 126]{tacha3}), and so $f_{r} \in  L^{1}\left(Q\times\Omega; W^{1}_{\#}L^{\Phi}_{per}(Y)\right)$. By \cite[Remark 2]{tacha2},	 it follows that $f_{r} \in L^{\Phi}\left(Q\times\Omega; W^{1}_{\#}L^{\Phi}_{per}(Y)\right)$.
	Moreover we have
	\begin{equation*}
		\dfrac{\partial f_{r}}{\partial y_{i}}(x,\omega,\cdot) =  v_{i}(x,\omega,\cdot) - \dfrac{1}{|B_{r}|}\int_{B_{r}}v_{i}(x, \omega, \cdot+\rho) d\rho, \quad 1\leq i \leq N.
	\end{equation*}
	As the function $v_{i}$ is periodic in its third argument $y$, we have that, when $r \to +\infty$, the right-hand side of the above equality goes towards $v_{i}(x,\omega,\cdot) - \mathcal{M}(v_{i}(x,\omega,\cdot))$, whence the existence of a unique function $u_{2}(x,\omega,\cdot) \in W^{1}_{\#}L^{\Phi}_{per}(Y)$ such that 
	\begin{eqnarray*}
		D_{y}u_{2}(x,\omega,\cdot) = \mathbf{v}(x,\omega,\cdot) - \mathcal{M}(\mathbf{v}(x,\omega,\cdot)), \quad dx\times d\mu-\textup{a.e.} \;\, (x,\omega) \in Q\times \Omega,
	\end{eqnarray*}
	hence the existence of a unique $u_{2} : Q\times \Omega \rightarrow W^{1}_{\#}L^{\Phi}_{per}(Y)$, $(x,\omega) \mapsto u_{2}(x,\omega,\cdot)$, lying in $L^{\Phi}(Q\times\Omega ; W^{1}_{\#}L^{\Phi}_{per}(Y))$ such that 
	\begin{equation}\label{lemc4}
		\mathbf{v} - \mathcal{M}(\mathbf{v}) = D_{y}u_{2}.
	\end{equation}	
	Let us finally derive the existence of $u_{1}$.\\ 
	Let	$\mathbf{\Psi}_{\epsilon}(x,\omega) = \varphi(x)\mathbf{\Psi}\left(T\left(\frac{x}{\epsilon}\right)\omega \right)$ with $\varphi \in \mathcal{C}^{\infty}_{0}(Q)$ and $\mathbf{\Psi} = (\psi_{j})_{1\leq j\leq N} \in \mathcal{V}^{\widetilde{\Phi}}_{\textup{div}}$ (i.e., $\textup{div}_{\omega,\widetilde{\Phi}}\mathbf{\Psi} = 0$). Clearly 
	\begin{equation*}
		\sum_{j=1}^{N} \iint_{Q\times \Omega} \dfrac{\partial u_{\epsilon}}{\partial x_{j}} \varphi\, \psi_{j}^{\epsilon}\, dxd\mu = - \sum_{j=1}^{N} \iint_{Q\times \Omega} u_{\epsilon }\psi_{j}^{\epsilon}\dfrac{\partial \varphi}{\partial x_{j}}\, dxd\mu,
	\end{equation*}
	where $\psi_{j}^{\epsilon}(x,\omega) = \psi_{j}\left(T\left(\frac{x}{\epsilon}\right)\omega \right)$. The passage to the limit  (when $E' \ni \epsilon \rightarrow 0$) yields 
	\begin{equation*}
		\sum_{j=1}^{N} \iiint_{Q\times \Omega\times Y} v_{j} \varphi\, \psi_{j}\, dxd\mu dy = - \sum_{j=1}^{N} \iiint_{Q\times \Omega\times Y} u_{0 }\psi_{j}\dfrac{\partial \varphi}{\partial x_{j}}\, dxd\mu dy,
	\end{equation*}
	or equivalently
	\begin{equation*}
		\iiint_{Q\times \Omega\times Y} \left( \textbf{v}(x,\omega,y) - Du_{0}(x,\omega)  \right)\cdot \mathbf{\Psi}(\omega)\varphi(x) \, dxd\mu dy = 0,
	\end{equation*}
	and so, as $\varphi$ is arbitrarily fixed in $\mathcal{C}^{\infty}_{0}(Q)$,
	\begin{equation*}
		\iint_{\Omega\times Y} \left( \mathbf{v}(x,\omega,y) - Du_{0}(x,\omega)  \right)\cdot \mathbf{\Psi}(\omega) \,d\mu dy = 0, \;\;\;  
	\end{equation*}
	for $\mathbf{\Psi} \in \mathcal{V}^{\widetilde{\Phi}}_{\textup{div}}$ and for a.e. $x\in Q$. \\
	This is also equivalent to
	\begin{equation*}
		\int_{\Omega} \left( \mathcal{M}( \mathbf{v}(x,\omega,\cdot)) - Du_{0}(x,\omega)  \right)\cdot \mathbf{\Psi}(\omega) \,d\mu = 0, \;\;\;  
	\end{equation*}
	for $\mathbf{\Psi} \in \mathcal{V}^{\widetilde{\Phi}}_{\textup{div}}$ and for a.e. $x\in Q$.\\
	Therefore, Lemma \ref{lem5} provides us with a unique $u_{1}(x, \cdot) \in  W^{1}_{\#}L^{\Phi}(\Omega)$ such that 
	\begin{equation}\label{lemc3}
		\mathcal{M}(\mathbf{v}(x,\cdot,\cdot)) - Du_{0}(x,\cdot) = \overline{D}_{\omega}u_{1}(x,\cdot) \;\; \textup{for} \;\, \mu\textup{-a.e.} \;\, x \in Q.
	\end{equation} 
	The above equality reads as, for a.e. $x\in Q$, $	\mathcal{M}(\mathbf{v}(x,\omega,\cdot)) - Du_{0}(x,\omega) = \overline{D}_{\omega}u_{1}(x,\omega)$, for $\mu-$a.e. $\omega\in \Omega$. \\
	Finally, (\ref{lemc4}) and (\ref{lemc3}) get $\mathbf{v} = Du_{0} + \overline{D}_{\omega}u_{1} + D_{y}u_{2}$.  On the other hand, it is shown in \cite[Theorem 12]{franck} that, $u_{1} \in L^{1}\left(Q; W^{1}_{\#}L^{\Phi}(\Omega)\right)$ and  that, if $\widetilde{\Phi} \in \Delta'$ then $u_{1} \in L^{\Phi}\left(Q; W^{1}_{\#}L^{\Phi}(\Omega)\right)$. The proof is complete.
\end{proof}
\begin{remark}
	Under the hypothesis of Theorem \ref{lem29}, if the sequence $(u_{\epsilon})_{\epsilon\in E} \subset  L^{\Phi}(Q\times\Omega)$ is such that $u_{\epsilon}(\cdot,\omega) \in X$ for all $\epsilon\in E$ and for $\mu$-a.e. $\omega\in \Omega$, where $X$ is a norm closed convex subset of $W^{1}L^{\Phi}(Q)$, then as in \cite[Theorem 2]{sango} the stochastic two-scale limit $u_{0}$ is such that $u_{0}(\cdot,\omega) \in X$ for $\mu$-a.e. $\omega\in \Omega$.
\end{remark}


\section{Application to the homogenization of a class of minimization problems}\label{sect4}

In this section, we intend to study the asymptotic behaviour (as $0 < \epsilon \rightarrow 0$) of the sequence of solutions to the minimization problem \eqref{ras2}-\eqref{ras1} under the hypotheses of Subsection \ref{hypoproblem}, in particular when the random-periodic integrand $f$ in \eqref{ras1}  satisfying the conditions $(H_{1})-(H_{4})$. For that, we first establish the existence and uniqueness of the solution to the problem \eqref{ras2} for each $\epsilon>0$ given. Next, we establish the preliminary convergence results for the problem \eqref{ras1}. To end, we regularize the integrand $f$ in order to prove the main homogenization result stated in Theorem \ref{lem37}. However, we make the following remarks regarding the originality of our homogenization problem.


\begin{remark}
	If we consider the particular dynamical system $T(x)$ on $\Omega = \mathbb{T}^{N} \equiv \mathbb{R}^{N}/\mathbb{Z}^{N}$ (the $N$-dimensional torus) defined by $T(x)\omega = x+\omega\;\textup{mod}\;\mathbb{Z}^{N}$, then one can view $\Omega$ as the unit cube in $\mathbb{R}^{N}$ with all the pairs of antipodal faces being identified. The Lebesgue measure on $\mathbb{R}^{N}$ induces the Haar measure on $\mathbb{T}^{N}$ which is invariant with respect to the action of $T(x)$ on $\mathbb{T}^{N}$. Moreover, $T(x)$ is ergodic and in this situation, any function on $\Omega$ may be regarded as a periodic function on $\mathbb{R}^{N}$ whose period in each coordinate is 1, so that in this case our integrand $f$ may be viewed as a periodic function with respect to both variables $\omega$ and $y$. Therefore, the stochastic-periodic homogenization problem (\ref{ras2}) is equivalent to the reiterated-periodic homogenization problem,
	\begin{equation*}
		\min \left\{ \int_{Q} f\left(x,\frac{x}{\epsilon},\frac{x}{\epsilon^{2}}, Dv(x) \right) dx \; : \; v \in  W^{1}_{0}L^{\Phi}(Q) \right\},
	\end{equation*}
	and whose a particular case is treated in \cite{tacha3}. 
\end{remark}
\begin{remark}
	On the other hand, if $\Phi(t) = \frac{t^{p}}{p}$ ($p>1, \;t\geq 0$), then $\Phi \in \Delta_{2}\cap\Delta'$ (with $\widetilde{\Phi}(t)=\frac{t^{q}}{q}$, $q=\frac{p}{p-1}$) and one has $L^{\Phi}(Q\times\Omega)\equiv L^{p}(Q\times\Omega)$, $L^{\Phi}(Q\times\Omega\times Y_{per}) \equiv L^{p}(Q\times\Omega; L^{p}_{per}(Y))$ and $W^{1}_{0}L^{\Phi}(Q; L^{\Phi}(\Omega))\equiv W^{1,p}_{0}(Q; L^{p}(\Omega))=L^{p}(\Omega; W^{1,p}_{0}(Q))$.  Therefore, in this case, the stochastic-periodic homogenization problem (\ref{ras2}) can be rewritten in the classical Sobolev's spaces as
	\begin{equation*}
		\min \left\{ \int_{Q\times\Omega} f\left(x,T\left(\frac{x}{\epsilon}\right)\omega,\frac{x}{\epsilon^{2}}, Dv(x,\omega) \right) dx \; : \; v \in  L^{p}(\Omega; W^{1,p}_{0}(Q)) \right\},
	\end{equation*}
	and which is treated in \cite{sango}.
\end{remark}

\subsection{Existence and uniqueness of minimizers for (\ref{ras2})}

Let us first note that, since the function $f$ in (\ref{ras1}) is convex in the last argument, it results (see, \cite[Lemma 13]{franck}) the following lemma.
\begin{lemma}\label{lem9}
	Let $(x,\omega,y,\lambda) \rightarrow g(x,\omega,y,\lambda)$ be a function from $\mathbb{R}^{N}\times\Omega\times \mathbb{R}^{N}\times \mathbb{R}^{N}$ into $\mathbb{R}$ satisfying the following properties:
	\begin{itemize}
		\item[(i)] $g(x,\cdot,y, \lambda)$ is measurable for all $(x,y,\lambda) \in \mathbb{R}^{N}\times\mathbb{R}^{N}\times \mathbb{R}^{N}$ ;
		\item[(ii)] $g(x,\omega,y,\cdot)$ is strictly convex for $d\mu\times dy$-almost all $(\omega,y) \in \Omega\times\mathbb{R}^{N}$ and for all $x\in \mathbb{R}^{N}$ ;
		\item[(iii)] There exists  $c_{0} > 0$ such that 
		\begin{equation} \label{lem8}
			| g(x,\omega,y,\lambda)| \leq c_{0}(1+ \Phi(|\lambda|)),
		\end{equation}
		for all $(x,\lambda) \in \mathbb{R}^{N}\times \mathbb{R}^{N}$ and for $d\mu\times dy$-almost all $(\omega,y) \in \Omega\times\mathbb{R}^{N}$.
	\end{itemize}
	Then $g$ is continuous in the last argument. Precisely, one has
	\begin{equation*}\label{lem7}
		| g(x,\omega,y,\lambda) - g(x,\omega,y,\nu) | \leq c'_{0} \dfrac{1 + \Phi(2(1+ |\lambda| + |\nu|))}{1+ |\lambda| + |\nu|} |\lambda - \nu |,
	\end{equation*}
	for all $x,\lambda,\nu \in \mathbb{R}^{N}$ and for $d\mu\times dy$-almost all $(\omega,y) \in \Omega\times\mathbb{R}^{N}$, with $c'_{0} = 2Nc_{0}$.
\end{lemma}
Let now $\mathbf{v} = (v_{i}) \in \left[\mathcal{C}^{\infty}(\Omega)\otimes \mathcal{C}(\overline{Q}) \right]^{N}$. By (\ref{lem8}), it is an easy task to check that the function $(\omega,y) \rightarrow g(x,\omega,y, \mathbf{v}(x,\omega))$ of $\Omega\times\mathbb{R}^{N}$ into $\mathbb{R}$ denoted by $g(x,\cdot,\cdot, \mathbf{v}(x,\cdot))$ belongs to $L^{\infty}(\Omega\times\mathbb{R}^{N})$ (see \cite[Section 4]{sango}).
Hence one can define the function $x \rightarrow g(x,\cdot,\cdot, \mathbf{v}(x,\cdot))$ of $\overline{Q}$ into $L^{\infty}(\Omega\times\mathbb{R}^{N})$ denoted by $g(\cdot,\cdot, \mathbf{v})$ as element of $\mathcal{C}(\overline{Q} ; L^{\infty}(\Omega\times\mathbb{R}^{N}))$.
Hence for each $\epsilon > 0$, the function $(x, \omega) \rightarrow g\left( x,T\left(\frac{x}{\epsilon}\right)\omega,\frac{x}{\epsilon^{2}}, \mathbf{v}\left( x,\omega\right)\right)$ of $Q\times\Omega$ into $\mathbb{R}$, denoted by $g^{\epsilon}(\cdot, \cdot, \mathbf{v})$ is well defined as an element of $L^{\infty}(Q\times \Omega)$, and arguing exactly as in \cite[Proposition 14 and Corollary 15]{franck} we have the following.
\begin{proposition}\label{lem14}
	Given $\epsilon >0$, the transformation $\mathbf{v} \rightarrow g^{\epsilon}(\cdot,\cdot, \mathbf{v})$ of $\left[ \mathcal{C}^{\infty}(\Omega)\otimes\mathcal{C}(\overline{Q}) \right]^{N}$ into $L^{\infty}(Q\times \Omega)$ extends by continuity to a mapping, still denoted by $\mathbf{v} \rightarrow g^{\epsilon}(\cdot,\cdot, \mathbf{v})$, of
	$\left[L^{\Phi}(Q\times \Omega) \right]^{N}$ into $L^{1}(Q\times \Omega)$ with the property 
	\begin{equation*}\label{lem12}
		\begin{array}{l}
			\parallel g^{\epsilon}(\cdot,\cdot, \mathbf{v}) - g^{\epsilon}(\cdot,\cdot, \mathbf{w}) \parallel_{L^{1}(Q\times \Omega)}  \\ 
			\leq c\, \left( \|1\|_{\widetilde{\Phi},Q\times\Omega} + \parallel \phi( 1+|\mathbf{v}|+ |\mathbf{w}|)\parallel_{\widetilde{\Phi},Q\times \Omega} \right)\parallel \mathbf{v}-\mathbf{w}  \parallel_{\left[L^{\Phi}(Q\times \Omega) \right]^{N}},
		\end{array}
	\end{equation*}
	for all $\mathbf{v}, \mathbf{w} \in \left[L^{\Phi}(Q\times \Omega) \right]^{N}$, with some constant $c > 0$.
\end{proposition}
\begin{corollary}\label{lem22}
	Under the hypothesis $(H_{1})-(H_{4})$, given $v \in 	W^{1}L^{\Phi}_{D_{x}}(Q; L^{\Phi}(\Omega))$, the function $(x,\omega) \rightarrow f\left( x,T\left(\frac{x}{\epsilon}\right)\omega,\frac{x}{\epsilon^{2}}, Dv(x,\omega)\right)$ of  $Q\times \Omega$ into $\mathbb{R}$, denoted $f^{\epsilon}(\cdot, \cdot, Dv)$ is well defined as an element of $L^{1}(Q\times \Omega)$. Moreover, we have 
	\begin{equation}\label{lem15}
		c_{1} \parallel Dv \parallel_{L^{\Phi}(Q\times\Omega)^{N}} \leq \parallel f^{\epsilon}(\cdot,\cdot, Dv) \parallel_{L^{1}(Q\times\Omega)} \leq c'_{2}(1 + \parallel Dv \parallel_{L^{\Phi}(Q\times\Omega)^{N}}),
	\end{equation}	
	for all $v \in W^{1}_{0}L^{\Phi}_{D_{x}}(Q; L^{\Phi}(\Omega))$, where $c'_{2} = c_{2}\max(1, |Q|)$ with $|Q| = \int_{Q} dx$. 
\end{corollary}
We are now able to prove the existence of a minimizer on $W^{1}_{0}L^{\Phi}_{D_{x}}(Q; L^{\Phi}(\Omega))$ of integral functional $F_{\epsilon}$ [see (\ref{ras1})], for each $\epsilon > 0$. Clearly, we have the following theorem.
\begin{theorem}\label{unitsolution}
	For $\epsilon > 0$, there exists a unique $u_{\epsilon} \in W^{1}_{0}L^{\Phi}_{D_{x}}(Q; L^{\Phi}(\Omega))$ that minimizes $	F_{\epsilon}$ on $W^{1}_{0}L^{\Phi}_{D_{x}}(Q; L^{\Phi}(\Omega))$, i.e.,
	\begin{equation*}
		F_{\epsilon}(u_{\epsilon}) = \min \left\{ F_{\epsilon}(v) : v \in W^{1}_{0}L^{\Phi}_{D_{x}}(Q; L^{\Phi}(\Omega)) \right\}.
	\end{equation*}
\end{theorem}
\begin{proof}
	Let $\epsilon > 0$ be fixed. Thanks to Proposition \ref{lem14} (with $g=f$), there is a constant $c > 0$ such that 
	\begin{equation*}
		\begin{array}{l}
			\left|  F_{\epsilon}(v) - F_{\epsilon}(w) \right| \\
			\leq c\, \left( \|1\|_{\widetilde{\Phi},Q\times\Omega} + \parallel \phi( 1+|Dv|+ |Dw|)\parallel_{\widetilde{\Phi},Q\times \Omega} \right)\parallel Dv-Dw  \parallel_{\left[L^{\Phi}(Q\times \Omega) \right]^{N}},
		\end{array}
	\end{equation*} 
	for all $v,w \in W^{1}_{0}L^{\Phi}_{D_{x}}(Q; L^{\Phi}(\Omega))$, so that $F_{\epsilon}$ is continuous. With this in mind, since $F_{\epsilon}$ is strictly convex (see $(H_{2})$) and coercive (see the left-hand side inequality in (\ref{lem15})), there exists a unique $u_{\epsilon}$ that minimizes $F_{\epsilon}$ on $W^{1}_{0}L^{\Phi}_{D_{x}}(Q; L^{\Phi}(\Omega))$.
\end{proof}  

\subsection{Preliminary convergence results} 

Let $\mathbf{v}  \in \left[ \mathcal{C}^{\infty}(\Omega)\otimes\mathcal{C}(\overline{Q} ; \mathcal{C}_{per}(Y)) \right]^{N}$. \\
The function $(\omega,y) \rightarrow f(x,\omega, y, \mathbf{v}(x,\omega,y))$ of $\Omega\times\mathbb{R}^{N}$ into $\mathbb{R}$ denoted by $f(x,\cdot,\cdot, \mathbf{v}(x,\cdot,\cdot))$ belongs to $L^{\infty}(\Omega\times\mathbb{R}^{N})$.   This being, using Lemma \ref{lem9}, that the function
$x \rightarrow f(x,\cdot,\cdot, \mathbf{v}(x,\cdot,\cdot))$ of $\overline{Q}$ into $L^{\infty}(\Omega\times\mathbb{R}^{N})$ denoted by $f(\cdot,\cdot, \mathbf{v})$ belongs to $\mathcal{C}(\overline{Q}; L^{\infty}(\Omega\times\mathbb{R}^{N}))$, with
\begin{eqnarray}\label{lem19}
	|f(\cdot,\cdot, \mathbf{v}) - f(\cdot,\cdot, \mathbf{w})| \leq c'_{2} \left(  1 + \dfrac{\Phi(2(1+ |\mathbf{v}| + |\mathbf{w}|))}{1+ |\mathbf{v}| + |\mathbf{w}|}\right) |\mathbf{v} - \mathbf{w} |  \\ 
	\quad \textup{a.e. \, in} \, \overline{Q}\times \Omega\times Y,  \nonumber
\end{eqnarray}
for all $\mathbf{v}, \mathbf{w} \in \left[ \mathcal{C}^{\infty}(\Omega)\otimes\mathcal{C}(\overline{Q} ; \mathcal{C}_{per}(Y)) \right]^{N}$.
Hence for each $\epsilon > 0$, the function \\ $(x, \omega) \rightarrow f\left( x,T\left(\frac{x}{\epsilon}\right)\omega,\frac{x}{\epsilon^{2}}, \mathbf{v}\left( x,T\left(\frac{x}{\epsilon}\right)\omega,\frac{x}{\epsilon^{2}}\right)\right)$ of $Q\times\Omega$ into $\mathbb{R}$, denoted by $f^{\epsilon}(\cdot, \cdot, \mathbf{v}^{\epsilon})$ is well defined as an element of $L^{\infty}(Q\times \Omega)$.
Moreover, we have in view of $(H_{4})$,
\begin{equation*}
	f(\cdot,\cdot, \mathbf{v}) \in \mathcal{C}(\overline{Q}; L^{\infty}(\Omega ; L^{\infty}_{per}(Y))) \quad \textup{for \, all} \; \mathbf{v} \in \left[ \mathcal{C}^{\infty}(\Omega)\otimes\mathcal{C}(\overline{Q} ; \mathcal{C}_{per}(Y)) \right]^{N}.
\end{equation*}
We can now state and prove the first preliminary result.
\begin{proposition}\label{lem23}
	Suppose that $(H_{1})-(H_{4})$ hold. \\ For every $\mathbf{v} \in \left[ \mathcal{C}^{\infty}(\Omega)\otimes\mathcal{C}(\overline{Q} ; \mathcal{C}_{per}(Y)) \right]^{N}$ one has 
	\begin{equation}\label{lem17}
		\begin{array}{l}
			\displaystyle	\lim_{\epsilon \to 0} \iint_{Q\times\Omega} f\left( x,T\left(\frac{x}{\epsilon}\right)\omega,\frac{x}{\epsilon^{2}}, \mathbf{v}\left( x,T\left(\frac{x}{\epsilon}\right)\omega,\frac{x}{\epsilon^{2}}\right)\right) dxd\mu \\ 
			= \displaystyle \iiint_{Q\times\Omega\times Y} f(x,\omega,y, \mathbf{v}(x,\omega,y)) dxd\mu dy.
		\end{array}
	\end{equation}
	Furthermore, the mapping  $\mathbf{v} \rightarrow f(\cdot,\cdot, \mathbf{v})$ of $\left[ \mathcal{C}^{\infty}(\Omega)\otimes\mathcal{C}(\overline{Q} ; \mathcal{C}_{per}(Y)) \right]^{N}$ into $L^{1}(Q\times \Omega ; L^{1}_{per}(Y) )$ extends by continuity to a mapping, still denoted by $\mathbf{v} \rightarrow f(\cdot,\cdot, \mathbf{v})$, of
	$\left[L^{\Phi}(Q\times \Omega\times Y_{per}) \right]^{N}$ into $L^{1}(Q\times \Omega; L^{1}_{per}(Y))$ with the property 
	\begin{equation*}
		\begin{array}{l}
			\parallel f(\cdot, \cdot, \mathbf{v}) - f(\cdot, \cdot, \mathbf{w}) \parallel_{L^{1}(Q\times \Omega\times Y)} \\
			\leq c\, \left( \|1\|_{\widetilde{\Phi},Q\times\Omega\times Y} + \parallel \phi( 1+|\mathbf{v}|+ |\mathbf{w}|)\parallel_{\widetilde{\Phi},Q\times \Omega\times Y} \right)\parallel \mathbf{v}-\mathbf{w}  \parallel_{\left[L^{\Phi}(Q\times \Omega\times Y) \right]^{N}},
		\end{array}
	\end{equation*}
	for all $\mathbf{v}, \mathbf{w} \in \left[L^{\Phi}(Q\times \Omega\times Y_{per}) \right]^{N}$.
\end{proposition}
\begin{proof}
	Let	$\mathbf{v} \in \left[ \mathcal{C}^{\infty}(\Omega)\otimes\mathcal{C}(\overline{Q} ; \mathcal{C}_{per}(Y)) \right]^{N}$. Since $f(\cdot,\cdot, \mathbf{v})$ belongs to \\ $\mathcal{C}(\overline{Q}; L^{\infty}(\Omega ; L^{\infty}_{per}(Y)))$ and the sequence $(u_{\epsilon})$ defined by $u_{\epsilon}(x,\omega) =1$, a.e. $x\in Q$ and $\omega \in \Omega$, is weakly 2s-convergent in $L^{\Phi}(Q\times \Omega)$ to 1, the convergence result (\ref{lem17}) is a consequence of Remark \ref{lem16}($R_{1}$). \\
	On the other hand, since $\Phi \in \Delta_{2}$, using the Proposition \ref{lem14}, there is a constant $c= c(c_{2},Q,\Omega,\Phi)$ such that  
	\begin{equation*}
		\begin{array}{l}
			\parallel f(\cdot, \cdot, \mathbf{v}) - f(\cdot, \cdot, \mathbf{w}) \parallel_{L^{1}(Q\times \Omega\times Y)} \\
			\leq c\, \left( \|1\|_{\widetilde{\Phi},Q\times\Omega\times Y} + \parallel \phi( 1+|\mathbf{v}|+ |\mathbf{w}|)\parallel_{\widetilde{\Phi},Q\times \Omega\times Y} \right)\parallel \mathbf{v}-\mathbf{w}  \parallel_{\left[L^{\Phi}(Q\times \Omega\times Y) \right]^{N}},
		\end{array}
	\end{equation*}
	for all $\mathbf{v}, \mathbf{w} \in \left[ \mathcal{C}^{\infty}(\Omega)\otimes\mathcal{C}(\overline{Q} ; \mathcal{C}_{per}(Y)) \right]^{N}$. We end the proof by routine argument of continuity and density.
\end{proof}
\begin{corollary}\label{lem24}
	Let 
	\begin{equation*}
		\phi_{\epsilon}(x,\omega) = \psi_{0}(x,\omega) + \epsilon \psi_{1}\left(x,T\left(\frac{x}{\epsilon}\right)\omega\right) + \epsilon^{2} \psi_{2}\left(x,T\left(\frac{x}{\epsilon}\right)\omega,\frac{x}{\epsilon^{2}}\right),
	\end{equation*}
	with $\epsilon>0$, $(x,\omega) \in Q\times \Omega$, $\psi_{0} \in \mathcal{C}^{\infty}_{0}(Q)\otimes I^{\Phi}_{nv}(\Omega)$, $\psi_{1} \in \mathcal{C}^{\infty}_{0}(Q)\otimes \mathcal{C}^{\infty}(\Omega)$ and  $\psi_{2} \in \mathcal{C}^{\infty}_{0}(Q)\otimes \mathcal{C}^{\infty}(\Omega)\otimes \mathcal{C}_{per}^{\infty}(Y)$. Then,
	\begin{equation*}
		\begin{array}{l}
			\displaystyle 	\lim_{\epsilon \to 0} \iint_{Q\times\Omega} f\left(x,T\left(\frac{x}{\epsilon}\right)\omega,\frac{x}{\epsilon^{2}}, D\phi_{\epsilon}(x,\omega)\right) dxd\mu  \\ 
			= \displaystyle \iiint_{Q\times\Omega\times Y} f\left(x,\omega,y, D_{x}\psi_{0}(x,\omega)+ D_{\omega}\psi_{1}(x,\omega) + D_{y}\psi_{2}(x,\omega,y)\right) dxd\mu dy.
		\end{array}
	\end{equation*}
\end{corollary}
\begin{proof}
	As $D\psi_{1}\left(x,T\left(\frac{x}{\epsilon}\right)\omega\right) = (D_{x}\psi_{1})^{\epsilon} + \frac{1}{\epsilon}(D_{\omega}\psi_{1})^{\epsilon}$ and  \\ $D\psi_{2}\left(x,T\left(\frac{x}{\epsilon}\right)\omega,\frac{x}{\epsilon^{2}}\right) = (D_{x}\psi_{2})^{\epsilon} + \frac{1}{\epsilon}(D_{\omega}\psi_{2})^{\epsilon} + \frac{1}{\epsilon^{2}}(D_{y}\psi_{2})^{\epsilon}$, we have
	\begin{equation}\label{tc1}
		D\phi_{\epsilon}(x,\omega) = D_{x}\psi_{0}+ \epsilon (D_{x}\psi_{1})^{\epsilon} + (D_{\omega}\psi_{1})^{\epsilon} + \epsilon^{2}(D_{x}\psi_{2})^{\epsilon} + \epsilon (D_{\omega}\psi_{2})^{\epsilon} +  (D_{y}\psi_{2})^{\epsilon}.
	\end{equation}
	Recalling that functions $D_{x}\psi_{0}, D_{x}\psi_{1}$, $ D_{\omega}\psi_{1}$ belongs to $\left[\mathcal{C}^{\infty}(\Omega)\otimes \mathcal{C}(\overline{Q}, \mathbb{R}) \right]^{N}$ and $D_{y}\psi_{2}$ belong to $\left[ \mathcal{C}^{\infty}(\Omega)\otimes\mathcal{C}(\overline{Q} ; \mathcal{C}_{per}(Y)) \right]^{N}$, the function $f(\cdot,\cdot, D_{x}\psi_{0}+ D_{\omega}\psi_{1} + D_{y}\psi_{2}) $ belongs to  $\mathcal{C}(\overline{Q}; L^{\infty}(\Omega ; L^{\infty}_{per}(Y)))$, so that (\ref{lem17}) implies 
	\begin{eqnarray}\label{lem21}
		\lim_{\epsilon \to 0} \iint_{Q\times\Omega} f^{\epsilon}(\cdot,\cdot, D_{x}\psi_{0}+ (D_{\omega}\psi_{1})^{\epsilon} + (D_{y}\psi_{2})^{\epsilon}) dxd\mu \nonumber \\
		= \iiint_{Q\times\Omega\times Y} f\left(x,\omega,y, D_{x}\psi_{0}+ D_{\omega}\psi_{1} + D_{y}\psi_{2}\right) dxd\mu dy.
	\end{eqnarray}
	On the other hand, since (\ref{tc1}) holds, it follows from (\ref{lem19})
	\begin{equation*}
		|f^{\epsilon}(\cdot,\cdot, D\phi_{\epsilon}) - f^{\epsilon}(\cdot,\cdot, D_{x}\psi_{0}+ (D_{\omega}\psi_{1})^{\epsilon} + (D_{y}\psi_{2})^{\epsilon}) | \leq c\epsilon \quad \textup{in} \; Q\times\Omega\times Y, \; \epsilon >0,
	\end{equation*} 
	where $c = c\left( \|D_{x}\psi_{1}\|_{\infty}, \phi(\|D_{x}\psi_{0}\|_{\infty}), \phi(\|D_{\omega}\psi_{1}\|_{\infty}), \phi(\|D_{y}\psi_{2}\|_{\infty})  \right) >0$. Hence 
	\begin{equation}\label{lem20}
		f^{\epsilon}(\cdot,\cdot, D\phi_{\epsilon}) - f^{\epsilon}(\cdot,\cdot, D_{x}\psi_{0}+ (D_{\omega}\psi_{1})^{\epsilon} + (D_{y}\psi_{2})^{\epsilon}) \rightarrow 0 \quad \textup{in} \; L^{1}(Q\times \Omega) \; \textup{as} \; \epsilon \to 0.
	\end{equation}
	The proof is completed by combining (\ref{lem21})-(\ref{lem20}) with the decomposition 
	\begin{equation*}
		\begin{array}{l}
			\displaystyle 	\iint_{Q\times\Omega}  f^{\epsilon}(\cdot,\cdot, D\phi_{\epsilon})dxd\mu - 	\displaystyle \iiint_{Q\times\Omega\times Y} f(\cdot,\cdot, D_{x}\psi_{0}+ D_{\omega}\psi_{1} + D_{y}\psi_{2}) dxd\mu dy \\
			= \displaystyle \iint_{Q\times\Omega} \left[  f^{\epsilon}(\cdot,\cdot, D\phi_{\epsilon}) -  f^{\epsilon}(\cdot,\cdot, D_{x}\psi_{0}+ (D_{\omega}\psi_{1})^{\epsilon} + (D_{y}\psi_{2})^{\epsilon}) 
			\right] dxd\mu \\
			\;\,	 + \displaystyle \iint_{Q\times\Omega}  f^{\epsilon}(\cdot,\cdot, D_{x}\psi_{0}+ (D_{\omega}\psi_{1})^{\epsilon} + (D_{y}\psi_{2})^{\epsilon}) dxd\mu \\
			\;\,	  - \displaystyle \iiint_{Q\times\Omega\times Y} f(\cdot,\cdot, D_{x}\psi_{0}+ D_{\omega}\psi_{1} + D_{y}\psi_{2}) dxd\mu dy. 
		\end{array}
	\end{equation*}
\end{proof}
Now, observing Theorem \ref{lem29} we set,  
\begin{equation*}
	\mathbb{F}_{0}^{1}L^{\Phi} = W^{1}_{0}L^{\Phi}_{D_{x}}(Q; I_{nv}^{\Phi}(\Omega))\times L^{\Phi}_{\overline{D}_{\omega}}(Q; W^{1}_{\#}L^{\Phi}(\Omega))\times L^{\Phi}_{D_{y}}(Q\times\Omega; W^{1}_{\#}L^{\Phi}_{per}(Y)), 
\end{equation*}
where $L^{\Phi}_{\overline{D}_{\omega}}(Q; W^{1}_{\#}L^{\Phi}(\Omega)) = \left\{ u \in L^{1}(Q; W^{1}_{\#}L^{\Phi}(\Omega)) ; \overline{D}_{\omega}u \in L^{\Phi}(Q\times \Omega)^{N}  \right\}$ and \\ $L^{\Phi}_{D_{y}}(Q\times\Omega; W^{1}_{\#}L^{\Phi}_{per}(Y)) = \left\{ u \in L^{1}(Q\times\Omega; W^{1}_{\#}L^{\Phi}_{per}(Y)) ; D_{y}u \in L^{\Phi}(Q\times \Omega\times Y_{per})^{N}  \right\}$. \\
Hence, when $\widetilde{\Phi}$ is of class $\Delta'$ then
\begin{equation}\label{banac1}
	\mathbb{F}_{0}^{1}L^{\Phi} = W^{1}_{0}L^{\Phi}_{D_{x}}(Q; I_{nv}^{\Phi}(\Omega))\times L^{\Phi}(Q; W^{1}_{\#}L^{\Phi}(\Omega))\times L^{\Phi}(Q\times\Omega; W^{1}_{\#}L^{\Phi}_{per}(Y)).
\end{equation}
We equip $\mathbb{F}_{0}^{1}L^{\Phi}$ with the norm 
\begin{equation*}
	\parallel \mathbf{u} \parallel_{\mathbb{F}_{0}^{1}L^{\Phi}} = \parallel Du_{0} \parallel_{L^{\Phi}(Q\times\Omega)^{N}} + \parallel \overline{D}_{\omega}u_{1} \parallel_{L^{\Phi}(Q\times\Omega)^{N}} + \parallel D_{y}u_{2} \parallel_{L^{\Phi}(Q\times\Omega\times Y)^{N}},
\end{equation*} 
where $\mathbf{u} = (u_{0}, u_{1}, u_{2}) \in \mathbb{F}_{0}^{1}L^{\Phi}$.
With this norm, $\mathbb{F}_{0}^{1}L^{\Phi}$ in (\ref{banac1}) is a Banach space admitting 
\begin{equation}\label{banac2}
	F^{\infty}_{0} = \left[  \mathcal{C}^{\infty}_{0}(Q)\otimes I_{nv}^{\Phi}(\Omega) \right] \times \left[ \mathcal{C}^{\infty}_{0}(Q)\otimes I_{\Phi}(\mathcal{C}^{\infty}(\Omega)) \right] \times \left[  \mathcal{C}^{\infty}_{0}(Q)\otimes \mathcal{C}^{\infty}(\Omega)\otimes D_{\#}(Y) \right],	
\end{equation}
as a dense subspace, where $I_{\Phi}$ denotes the canonical mapping of $\mathcal{C}^{\infty}(\Omega)$ into the completion $W^{1}_{\#}L^{\Phi}(\Omega)$ and $D_{\#}(Y) = \left\{ u \in \mathcal{C}^{\infty}_{per}(Y) : \mathcal{M}_{Y}(u) =0 \right\}$. \\
Let now $\mathbf{v} = (v_{0}, v_{1}, v_{2}) \in \mathbb{F}_{0}^{1}L^{\Phi}$, set $\mathbb{D}\mathbf{v} = Dv_{0} + \overline{D}_{\omega}v_{1} + D_{y}v_{2} \in L^{\Phi}(Q\times\Omega\times Y)^{N}$ and define the functional $F$ on $\mathbb{F}_{0}^{1}L^{\Phi}$ by 
\begin{equation*}
	F(\mathbf{v}) = \iiint_{Q\times\Omega\times Y} f(\cdot,\cdot, \mathbb{D}\mathbf{v})\, dxd\mu dy,
\end{equation*}
where the function $f$ here is defined as in Proposition \ref{lem23}. \\
The hypotheses $(H_{1})-(H_{4})$ drive to the following lemma.
\begin{lemma}
	There exists a unique $\mathbf{u} = (u_{0}, u_{1}, u_{2}) \in \mathbb{F}_{0}^{1}L^{\Phi}$ [see (\ref{banac1})] such that 
	\begin{equation}\label{lem28}
		F(\mathbf{u}) = \inf \left\{  F(\mathbf{v})\, : \, \mathbf{v} \in \mathbb{F}_{0}^{1}L^{\Phi} \right\}.
	\end{equation}
\end{lemma}

\subsection{Regularization} 

Following \cite[Subsection 4.3]{franck}, we regularize the integrand $f$ in order to get an approximating family of integrands $(f_{n})_{n\in\mathbb{N}^{\ast}}$ having in particular some properties $(H_{1})-(H_{4})$. Precisely, let $\theta_{n} \in \mathcal{C}^{\infty}_{0}(\mathbb{R}^{N})$ with $0\leq \theta_{n}$, $\textup{supp}\theta_{n} \subset \frac{1}{n}\overline{B_{N}}$ (where $\overline{B_{N}}$ denotes the closure of the open unit ball $B_{N}$ in $\mathbb{R}^{N}$) and $\int \theta_{n}(\eta)d\eta = 1$. Setting 
\begin{equation*}
	f_{n}(x,\omega,y, \lambda) = \int \theta_{n}(\eta)f(x,\omega,y, \lambda-\eta)d\eta, \quad (x,\omega,y,\lambda) \in \mathbb{R}^{N}\times\Omega\times \mathbb{R}^{N}\times \mathbb{R}^{N}.
\end{equation*}
The main properties of this new integrand are the following:
\begin{itemize}
	\item[$(H_{1})_{n}$] for all $(x,\lambda) \in \mathbb{R}^{N}\times \mathbb{R}^{N}$ and for almost all $(\omega,y)\in\Omega\times\mathbb{R}^{N}$, $f_{n}(x,\cdot,\cdot, \lambda)$ is measurable and $f_{n}(\cdot,\omega,\cdot,\cdot)$ is continuous. Moreover, there exist a continuous positive function $\varpi : \mathbb{R} \to \mathbb{R}_{+}$ with $\varpi(0)=0$, and a function $a \in L^{1}_{loc}(\mathbb{R}^{N}_{y})$ such that 
	\begin{equation*}
		|f_{n}(x,\omega,y,\lambda) - f_{n}(x',\omega,y,\lambda)| \leq \varpi(|x - x'|)[a(y) + f_{n}(x,\omega,y,\lambda)],
	\end{equation*}
	for all $x, x' \in \mathbb{R}^{N}$, $\lambda \in \mathbb{R}^{N}$ and for $d\mu\times dy$-almost all $(\omega,y) \in \Omega\times\mathbb{R}^{N}$ ;
	\item[$(H_{2})_{n}$] $f_{n}(x,\omega,y,\cdot)$ is strictly convex for almost all $\omega \in \Omega$ and for all $x,y \in \mathbb{R}^{N}$  ;
	\item[$(H_{3})_{n}$] There is constant $c_{5} > 0$ such that 
	\begin{equation*}
		f_{n}(x,\omega,y,\lambda) \leq c_{5}(1+ \Phi(|\lambda|)),
	\end{equation*}
	for all $x,y,\lambda \in \mathbb{R}^{N}$ and for almost all $\omega \in \Omega$ ;
	\item[$(H_{4})_{n}$] $f_{n}(x,\omega,\cdot, \lambda)$ is periodic for all $x,\lambda \in \mathbb{R}^{N}$ and for almost all $\omega \in \Omega$;
	\item[$(H_{5})_{n}$] $\dfrac{\partial f_{n}}{\partial \lambda}(x,\omega,y,\lambda)$ exists for all $x,y,\lambda \in \mathbb{R}^{N}$ and for almost all $\omega\in \Omega$, and there exists a constant $c_{6}=c_{6}(n) >0$ such that 
	\begin{equation*}
		\left| \dfrac{\partial f_{n}}{\partial \lambda}(x,\omega,y,\lambda) \right| \leq c_{6} (1+\phi(|\lambda|)).
	\end{equation*}
\end{itemize}
That being the case, we obtain the results in Proposition \ref{lem23} and in Corollary \ref{lem24}, where $f$ is replaced by $f_{n}$, and arguing as in the proof of \cite[Lemma 20]{franck} we have the following lemma.
\begin{lemma}
	For every $\mathbf{v} \in \left[L^{\Phi}(Q\times\Omega\times Y_{per})\right]^{N}$, as $n\to \infty$, one has
	\begin{equation*}
		f_{n}(\cdot,\cdot,\mathbf{v}) \rightarrow f(\cdot,\cdot,\mathbf{v}) \quad \textup{in} \; L^{1}(Q\times\Omega\times Y).
	\end{equation*}
\end{lemma} 
We are now ready to give one of the most important results of this section.
\begin{proposition}\label{conseq1}
	Let $(\mathbf{v}_{\epsilon})_{\epsilon}$ be a sequence in $ \left[L^{\Phi}(Q\times\Omega)\right]^{N}$ which weakly 2s-converges (in each component) to $\mathbf{v} \in \left[L^{\Phi}(Q\times\Omega\times Y_{per})\right]^{N}$. Then, for any integer $n\geq 1$,  we have that there exists a constant $C'$ such that
	\begin{equation*}
		\iiint_{Q\times\Omega\times Y} f_{n}(\cdot,\cdot, \mathbf{v}) dxd\mu dy  - \dfrac{C'}{n} \leq \lim_{\epsilon \to 0} \iint_{Q\times\Omega}  f\left(x, T\left(\frac{x}{\epsilon}\right)\omega,\frac{x}{\epsilon^{2}}, \mathbf{v}_{\epsilon}(x,\omega)\right)  dxd\mu.
	\end{equation*}
\end{proposition}
\begin{proof}
	Arguing as in the proof of \cite[Proposition 21]{franck} with the sequence \\ $(\mathbf{v}_{l})_{l} \in \left[ \mathcal{C}^{\infty}_{0}(Q)\otimes\mathcal{C}^{\infty}(\Omega)\otimes\mathcal{C}^{\infty}_{per}(Y) \right]^{N}$ such that $\mathbf{v}_{l} \rightarrow \mathbf{v}$ in  $\left[L^{\Phi}(Q\times\Omega\times Y_{per})\right]^{N}$ as $l\to \infty$. One get the result.
\end{proof} 

Letting $n\to\infty$ in Proposition \ref{conseq1}, and replacing $\mathbf{v}_{\epsilon}$ by $Du_{\epsilon}$, with $Du_{\epsilon}$ stochastically weakly two-scale converges (componentwise) to $\mathbb{D}\mathbf{u} = Du_{0} + \overline{D}_{\omega}u_{1} + D_{y}u_{2} $ in $\left[ L^{\Phi}(Q\times\Omega\times Y_{per}) \right]^{N}$, one obtains the following result.
\begin{corollary}\label{lem32}
	Let $(u_{\epsilon})_{\epsilon\in E}$ be a sequence in $W^{1}L^{\Phi}_{D_{x}}(Q; L^{\Phi}(\Omega))$. Assume that $(Du_{\epsilon})_{\epsilon\in E}$ stochastically weakly two-scale converges componentwise to $\mathbb{D}\mathbf{u} = Du_{0} + \overline{D}_{\omega}u_{1} + D_{y}u_{2} \in \left[ L^{\Phi}(Q\times\Omega\times Y_{per}) \right]^{N}$, where $\mathbf{u} = (u_{0}, u_{1}, u_{2}) \in \mathbb{F}^{1}_{0}L^{\Phi}$. Then, 
	\begin{eqnarray*}
		\iiint_{Q\times\Omega\times Y} f(x,\omega,y, \mathbb{D}\mathbf{u}(x,\omega,y)) dxd\mu dy \\
		\leq \lim_{\epsilon \to 0} \iint_{Q\times\Omega}  f\left(x, T\left(\frac{x}{\epsilon}\right)\omega,\frac{x}{\epsilon^{2}}, Du_{\epsilon}(x,\omega)\right)  dxd\mu.
	\end{eqnarray*}
\end{corollary}

\subsection{Proof of main homogenization result}

Our objective in this section is to prove the convergence result stated in Theorem \ref{lem37}. For the reader's convenience, we recall it here.

\noindent \textbf{Theorem \ref{lem37}.}(Main homogenization result) \\
\textit{	For each $\epsilon > 0$, let $(u_{\epsilon})_{\epsilon\in E} \in W^{1}_{0}L^{\Phi}_{D_{x}}(Q; L^{\Phi}(\Omega))$ be the unique solution of (\ref{ras2}). Then, as $\epsilon \to 0$, }
\begin{equation}\label{lem30}
	u_{\epsilon} \rightarrow u_{0} \quad stoch. \;\textup{in} \; L^{\Phi}(Q\times\Omega)-weak \,2s,  
\end{equation}
\textit{and }	 
\begin{equation}\label{lem31}
	Du_{\epsilon} \rightarrow Du_{0} + \overline{D}_{\omega}u_{1} + D_{y}u_{2} \quad stoch. \;\textup{in} \; L^{\Phi}(Q\times\Omega)^{N}-weak \,2s,  
\end{equation}
\textit{	where $\mathbf{u} = (u_{0}, u_{1}, u_{2}) \in \mathbb{F}_{0}^{1}L^{\Phi}$ [see (\ref{banac1})] is the unique solution to the minimization problem (\ref{lem28}). }

\begin{proof}
	In view of the growth conditions in $(H_{3})$, the sequence $(u_{\epsilon})_{\epsilon>0}$ is bounded in $W^{1}_{0}L^{\Phi}_{D_{x}}(Q; L^{\Phi}(\Omega))$ and so the sequence $(f^{\epsilon}(\cdot,\cdot, Du_{\epsilon}))_{\epsilon>0}$ is bounded in $L^{1}(Q\times\Omega)$. Thus, given an arbitrary fundamental sequence $E$, we get by Theorem \ref{lem29} the existence of a subsequence $E'$ from $E$ and a triplet $\mathbf{u} = (u_{0}, u_{1}, u_{2}) \in \mathbb{F}_{0}^{1}L^{\Phi}$ such that (\ref{lem30})-(\ref{lem31}) hold when $E' \ni \epsilon \to 0$. The sequence $(F_{\epsilon}(u_{\epsilon}))_{\epsilon>0}$ consisting of real numbers being bounded, since $(u_{\epsilon})_{\epsilon>0}$ is bounded in $W^{1}_{0}L^{\Phi}_{D_{x}}(Q; L^{\Phi}(\Omega))$, there exists a subsequence from $E'$ not relabeled such that $\lim_{E' \ni\epsilon \to 0} F_{\epsilon}(u_{\epsilon})$ exists. 
	We still have to verify that $\mathbf{u} = (u_{0}, u_{1}, u_{2})$ solves (\ref{lem28}). In fact, if $\mathbf{u}$ solves this problem, then thanks to the uniqueness of the solution of (\ref{lem28}), the whole sequence $(u_{\epsilon})_{\epsilon>0}$ will verify (\ref{lem30}) and (\ref{lem31}) when $\epsilon \to 0$. Thus, our only concern here is to check that $\mathbf{u}$ solves problem (\ref{lem28}). To this end, in view of Corollary \ref{lem32}, we have 
	\begin{equation}\label{lem34}
		\begin{array}{l}
			\displaystyle	\iiint_{Q\times\Omega\times Y} f(x,\omega,y, \mathbb{D}\mathbf{u}(x,\omega,y)) dxd\mu dy \\
			\leq \lim_{E'\ni\epsilon \to 0} \displaystyle \iint_{Q\times\Omega}  f\left(x, T\left(\frac{x}{\epsilon}\right)\omega,\frac{x}{\epsilon^{2}}, Du_{\epsilon}(x,\omega)\right)  dxd\mu.
		\end{array}
	\end{equation}
	On the other hand, let us establish an upper bound for 
	\begin{equation*}
		\iint_{Q\times\Omega}  f\left(x, T\left(\frac{x}{\epsilon}\right)\omega,\frac{x}{\epsilon^{2}}, Du_{\epsilon}(x,\omega)\right)  dxd\mu.
	\end{equation*}
	To do that, let $\phi = \left(\psi_{0}, I_{\Phi}(\psi_{1}), \psi_{2}\right) \in F^{\infty}_{0}$ [see (\ref{banac2})] with $\psi_{0} \in \mathcal{C}_{0}^{\infty}(Q)\otimes I_{nv}^{\Phi}(\Omega)$, $\psi_{1} \in \mathcal{C}_{0}^{\infty}(Q)\otimes \mathcal{C}^{\infty}(\Omega)$ and $\psi_{2} \in \mathcal{C}_{0}^{\infty}(Q)\otimes \mathcal{C}^{\infty}(\Omega)\times \mathcal{D}_{\#}(Y)$. 	Define $\phi_{\epsilon}$ as in Corollary \ref{lem24}. Since $u_{\epsilon}$ is the minimizer, one has 
	\begin{eqnarray*}
		\iint_{Q\times\Omega}  f\left(x, T\left(\frac{x}{\epsilon}\right)\omega,\frac{x}{\epsilon^{2}}, Du_{\epsilon}(x,\omega)\right)  dxd\mu  \\
		\leq \iint_{Q\times\Omega}  f\left(x, T\left(\frac{x}{\epsilon}\right)\omega,\frac{x}{\epsilon^{2}}, D\phi_{\epsilon}(x,\omega)\right)  dxd\mu.
	\end{eqnarray*}
	Thus, using Corollary \ref{lem24} we get 
	\begin{equation*}
		\begin{array}{l}
			\displaystyle	\lim_{E'\ni\epsilon \to 0}  \displaystyle \iint_{Q\times\Omega}  f\left(x, T\left(\frac{x}{\epsilon}\right)\omega,\frac{x}{\epsilon^{2}}, Du_{\epsilon}(x,\omega)\right)  dxd\mu \\
			\leq  \displaystyle \iiint_{Q\times\Omega\times Y}  f(\cdot,\cdot, D\psi_{0} + D_{\omega}\psi_{1} + D_{y}\psi_{2})  dxd\mu dy,
		\end{array}
	\end{equation*}
	for any $\phi \in F^{\infty}_{0}$, and by density, for all $\phi \in \mathbb{F}_{0}^{1}L^{\Phi}$. From which we get 
	\begin{eqnarray}\label{lem33}
		\lim_{E'\ni\epsilon \to 0}  \iint_{Q\times\Omega}  f\left(x, T\left(\frac{x}{\epsilon}\right)\omega,\frac{x}{\epsilon^{2}}, Du_{\epsilon}(x,\omega)\right)  dxd\mu  \nonumber \\ \leq  \inf_{\mathbf{v} \in \mathbb{F}_{0}^{1}L^{\Phi}} \iiint_{Q\times\Omega\times Y}  f(\cdot,\cdot, \mathbb{D}\mathbf{v})  dxd\mu dy.
	\end{eqnarray}
	Inequalities (\ref{lem34}) and (\ref{lem33}) yield 
	\begin{equation*}
		\iiint_{Q\times\Omega \times Y} f(x,\omega,y, \mathbb{D}\mathbf{u}) dxd\mu dy = \inf_{\mathbf{v} \in \mathbb{F}_{0}^{1}L^{\Phi}} \iiint_{Q\times\Omega\times Y}  f(\cdot,\cdot, \mathbb{D}\mathbf{v})  dxd\mu dy
	\end{equation*}
	i.e., (\ref{lem28}). The proof is complete.
\end{proof}


\section*{Acknowledgment}
The authors would like to thank the anonymous referee for his/her pertinent remarks, comments and suggestions.

{\small
	
	\bibliographystyle{abbrv}
	\bibliography{stoc_per_biblio1}
}

\end{document}